\documentclass{amsart}
\usepackage[mathscr]{eucal}
\usepackage{amsfonts, amsmath, amsthm, amssymb, latexsym}
\newtheorem{theorem}{Theorem}[section]
\newtheorem{lemma}[theorem]{Lemma}
\newtheorem{prop}[theorem]{Proposition}
\newtheorem{cor}[theorem]{Corollary}
\theoremstyle{definition}
\newtheorem{defi}[theorem]{Definition}
\newtheorem{example}[theorem]{Example}
\theoremstyle{remark}
\newtheorem{remark}[theorem]{Remark}

\numberwithin{equation}{section}
\newcommand{\onto}{\,\,\twoheadrightarrow\,\,}
\newcommand{\la}{\label}
\newcommand{\GL}{\mathrm{GL}}
\newcommand{\Ker}{\mathrm{Ker}}
\newcommand{\Res}{\mathrm{Res}}
\newcommand{\Spec}{\mathrm{Spec}}
\newcommand{\Mod}{\mathrm{Mod}}

\newcommand{\End}{\mathrm{End}}

\newcommand{\id}{\mathrm{Id}}
\newcommand{\into}{\hookrightarrow}

\def\ms#1{\mathcal{#1}}
\def\c{\mathbb{C}}
\def\e{\boldsymbol{e}}
\def\A{\mathcal{A}}
\def\a{\mathbb{A}}
\def\k{\mathbb{K}}
\def\Z{\mathbb{Z}}
\def\Q{\mathbb{Q}}
\def\W{\mathrm{Irr}(W)}
\def\QQ{\mathbf{Q}}
\def\N{\mathbb{N}}
\def\P{\mathcal{P}}
\def\T{\mathcal{T}}
\def\D{\mathcal{D}}
\def\O{\mathcal{O}}
\def\Oln{\mathcal{O}^{\rm ln}}
\def\h{\mathfrak{h}}
\def\g{\mathfrak{g}}

\def\n{\mathfrak{n}}
\def\r{\mathfrak{r}}

\def\grd{\mathrm{gr}}

\def\sl2{{\mathfrak{s}\mathfrak{l}}_2}
\def\vreg{V_{\rm{reg}}}
\def\Hreg{H_{\rm{reg}}}
\def\Mreg{M_{\rm{reg}}}
\def\Otor{\ms O_{\rm{tor}}}
\def\KZ{\mathtt{KZ}}
\def\QA{\mathcal{Q}}
\def\kz{\mathrm{kz}}
\def\H{\mathcal{H}}
\def\Ireg{I_{\rm{reg}}}
\def\Reg{\mathrm{Reg}}
\def\mod{\mathrm{mod}}

\begin{document}
%
%
%
\title{Quasi-invariants of Complex Reflection Groups}
\author{Yuri Berest}
\address{Department of Mathematics,
Cornell University, Ithaca, NY 14853-4201, USA}
\email{berest@math.cornell.edu}
\author{Oleg Chalykh}
\address{School of Mathematics, University of Leeds, Leeds LS2 9JT, UK}
\email{oleg@maths.leeds.ac.uk}
\thanks{The first author was partially supported by the NSF grants DMS 04-07502 and DMS 09-01570.}
\maketitle
\begin{abstract}
We introduce quasi-invariant polynomials for an arbitrary finite
complex reflection group $W$. Unlike in the Coxeter case, the space
of quasi-invariants of a given multiplicity is not, in general, an
algebra but a module $Q_k$ over the coordinate ring of a
(singular) affine variety $ X_k $. We extend the main results of
\cite{BEG} to this setting: in particular, we show that the variety
$ X_k $ and the module $ Q_k $ are Cohen-Macaulay, and the rings of
differential operators on $ X_k $ and $ Q_k $ are simple rings,
Morita equivalent to the Weyl algebra $ A_n(\c)\,$, where $ n = \dim\,X_k\,$.
Our approach relies on representation theory of complex Cherednik
algebras introduced in \cite{DO} and is parallel to that of
\cite{BEG}. As an application, we prove the existence of shift
operators for an arbitrary complex reflection group, confirming a
conjecture of Dunkl and Opdam \cite{DO}. Another result is a proof
of a conjecture of Opdam \cite{O}, concerning certain operations
(KZ twists) on the set of irreducible representations of $W$.
\end{abstract}

\section{Introduction}
\la{intro}
The notion of a quasi-invariant polynomial for a finite Coxeter group was introduced by A.~Veselov
and one of the authors in \cite{CV90}. Although quasi-invariants were natural
generalization of invariants, they first appeared in a slightly disguised form (as symbols of
commuting differential operators). More recently, the rings of quasi-invariants and
associated varieties have been studied by means of representation theory
\cite{FV, EG1, BEG} and found applications in other areas, including noncommutative algebra
\cite{BEG}, mathematical physics \cite{B, CFV, FV1} and combinatorics \cite{GW, GW1, BM}.

The aim of the present paper is to define quasi-invariants for an arbitrary complex reflection group and
give new applications. We begin with a brief overview of our definition, referring the reader to
Section~\ref{sec1} for details.
Let $ W $ be a finite complex reflection group acting in its reflection representation $ V $. Denote by
$\, \A = \{H\} \,$ the set of reflection hyperplanes of $ W $ and write $ W_H $ for the (pointwise) stabilizer of
$ H \in \A $ in $W$. Each $ W_H $ is a cyclic subgroup of $W$ of order $ n_H \ge 2 $, whose group algebra
$ \c W_H \subseteq \c W $  is spanned by the idempotents
$$
\e_{H,\, i} = \frac{1}{n_H}\,\sum_{w \in W_H} (\det\,w)^{-i} \, w \ , \quad
i = 0,\,1,\,\ldots\,,\,n_H-1\ ,
$$
where $\,\det:\,W \to \c^{\times} \,$ is the determinant character of $W$ on $V$.
The group $W$ acts naturally on the polynomial algebra $ \c[V]$, and the
invariant polynomials $ f \in \c[V]^W $ satisfy the equations
\begin{equation}
\label{INV} \e_{H, -i} (f) = 0 \ ,\quad i = 1, \ldots,\, n_H-1 \ .
\end{equation}
More precisely, we have that $\,f \in \c[V]^W \,$ if and only if \eqref{INV} hold for all $\,H \in \A \,$.

Now, to define quasi-invariants we relax the equations \eqref{INV} in the following way.
For each $ H \in \A $, we
fix a linear form $\,\alpha_H \in V^* $, such that $\,H = \Ker\,\alpha_H \,$, and choose $\, n_H - 1 \,$
non-negative multiplicities $\, k_{H,i} \in \Z \,$, assuming $\, k_{H,i} = k_{H',i}\,$
whenever $ H $ and $ H' $ are in the same orbit of $ W $ in $ \A $. Then, we replace \eqref{INV} by
\begin{equation}
\label{IN1} \e_{H, -i} (f) \equiv 0\ \mbox{mod}\,\langle\alpha_H
\rangle^{n_H k_{H,i}} \ ,\quad i = 1, \ldots,\, n_H-1 \ ,
\end{equation}
where $\,\langle\alpha_H \rangle\,$ is the ideal in $ \c[V] $
generated by $ \alpha_H $. Letting $\, k := \{k_{H,i}\} $,
we call $\,f \in \c[V]\,$ a {\it $k$-quasi-invariant} of $W$
if it satisfies \eqref{IN1} for all $\, H \in \A \,$. It is easy to see that
this agrees with the earlier definition of quasi-invariants in the Coxeter
case (cf. Example~\ref{Cox}); however, unlike in that case, the subspace
$ Q_k(W) \subseteq \c[V] $ of $k$-quasi-invariants is not necessarily a ring.
Still, $ Q_k(W) $ contains $ \c[V]^W $, and the following remarkable property holds.
\begin{theorem}
\la{chev} $\, Q_k(W) $ is a free module over $\c[V]^W$ of rank $|W|$.
\end{theorem}
Since $\, Q_0(W) = \c[V] \,$, Theorem~\ref{chev} can be viewed as a generalization of a
classic result of Chevalley and Serre (see \cite{Ch}); equivalently, it can be stated by
saying that $ Q_k(W) $ is a Cohen-Macaulay module. For the Coxeter groups, this was
conjectured by Feigin and Veselov in \cite{FV} and
proved, by different methods, in \cite{EG1} and \cite{BEG}. It is worth mentioning that the
elementary argument of \cite{Ch} and its refinement in \cite{Bo} (see {\it loc.~cit},
Ch.~V, \S\,5, Theorem~1) do not work for nonzero $ k $.

We will prove Theorem~\ref{chev} (in fact, the more precise Theorem~\ref{stq}) by extending the
approach of \cite{BEG}, which is based on representation theory of Cherednik algebras.
We will also generalize another important result of
\cite{BEG} concerning the ring $ \D(Q_k) $ of differential operators on quasi-invariants.
\begin{theorem}
\label{ma} $\, \D(Q_k) \,$ is a simple ring, Morita equivalent to $ \D(V) $.
\end{theorem}
By a general result of Van den Bergh \cite{vdB} (see also \cite{BN}), Theorem~\ref{ma} is actually a
strengthening of Theorem~\ref{chev}; in this paper, however, we will prove these two
theorems by independent arguments, without using \cite{vdB} and \cite{BN}.

Although most of the elementary properties of quasi-invariants generalize
easily to the complex case, the proofs of Theorem~\ref{chev} and Theorem~\ref{ma} do not.
A key observation of \cite{BEG} linking quasi-invariants $ Q_k $ to
the rational Cherednik algebra $ H_k $  is the fact
that $Q_k $ is a module over the {\it spherical} subalgebra $\, U_k = \e H_k \e \,$ of
$ H_k $, and $ U_k $ is isomorphic to the ring $ \D(Q_k)^W $ of invariant differential
operators on $Q_k $. We will see that a
similar result holds for an arbitrary complex reflection
group; however, unlike in the Coxeter case (cf. \cite{BEG}, Lemma~6.4), this can hardly
be proved by direct calculation, working with generators of $ U_k $. The problem is that the ring of invariants
$ \c[V]^W $ of a complex refection group contains no quadratic polynomial, which makes explicit calculations
with generators virtually impossible\footnote{In fact, skimming the
classification table in \cite{ST} shows that there is an exceptional complex group with minimal fundamental degree 
as large as $60$.}.
To remedy this problem, we will work with the Cherednik algebra itself, lifting quasi-invariants at the level
of $\c W$-valued polynomials. More precisely, in Section~\ref{sec3}, we will define quasi-invariants $ \QQ_k(\tau) $
with values in an {\it arbitrary} representation $ \tau $ of $W$ as a module over the Cherednik algebra $H_k$.
(Checking that $\QQ_k(\tau) $ is indeed an $H_k$-module is easy, since $ H_k $ is generated by linear forms
and first order (Dunkl) operators.) The main observation (Theorem~\ref{Qfat}) is that the usual quasi-invariants
$ Q_k $ are obtained by symmetrizing the $ \tau$-valued ones, $ \QQ_k(\tau) $, with $ \tau $ being the regular
representation $ \c W$.
The existence of a natural $ U_k $-module structure on $ Q_k $ is a simple consequence of this construction and
the fact that $ H_k $ and $ U_k $ are Morita equivalent algebras for integral $k$. As we will see in
Section~\ref{doq} (Proposition~\ref{is}),
the key isomorphism $\, U_k \cong \D(Q_k)^W $ also follows easily from this, and Theorem~\ref{ma} (see Section~\ref{Simple})
can then be proven similarly to \cite{BEG}.

In Section~\ref{ShOper}, we will use quasi-invariants to show the existence of Heckman-Opdam shift
operators for an arbitrary complex reflection group. In the Coxeter case, this result was
established by an elegant argument by G.~Heckman \cite{H}, using Dunkl operators.
Heckman's proof involves explicit calculations with second order invariant operators,
which do not generalize to the complex case (exactly for the reason mentioned above). Still,
Dunkl and Opdam \cite{DO} have managed to extend Heckman's construction to the infinite family
of complex groups of type $ G(m,p,N) $ and conjectured the existence of shift operators in general.
Theorem~\ref{shiftoper} proves this conjecture of \cite{DO}. The idea behind
the proof is to study symmetries of the family of quasi-invariants $ \{ \QQ_k(\tau) \} $ under certain
transformations of multiplicities $k$, which induce the identity at the level of spherical algebra.

Section~\ref{o} reviews the definition and basic properties of the category $ \O $ for rational Cherednik algebras.
This category was introduced and studied in \cite{DO}, \cite{BEG} and \cite{GGOR} as an analogue of the
eponymous category of representations of a semisimple complex Lie algebra. In Section~\ref{o}, we gather
together results on the
category $ \O $ needed for the rest of the paper. Most of these results are either directly borrowed or
can be deduced from the above references (in the last case, for reader's convenience, we provide proofs).

In Section~\ref{SF}, we develop some aspects of representation theory of Cherednik algebras, which may
be of independent interest. First, in Section~\ref{sfff}, we introduce a shift functor
$\,\T_{k \to k'}:\, \O_k \to \O_{k'}\,$, relating representation categories of Cherednik algebras with
different values of multiplicities. This functor is analogous to the Enright completion in Lie theory
(see \cite{J}) and closely related to other types of shift functors which have appeared in the literature.
Some of these relations will be discussed in Section~\ref{othershifts}.

Next, in Section~\ref{kzt}, we will study a certain family of permutations $\,\{\kz_k\}_{k \in \Z} \,$ on the set $ \W $
of (isomorphism classes of) irreducible representations of $W$. These permuations (called KZ twists) were orginally defined
by E.~Opdam in
terms of Knizhnik-Zamolodchikov equations and studied using the finite Hecke algebra $ \H_k(W) $ (see \cite{O1, O, O3}).
In \cite{O1}, Opdam explicitly described KZ twists for all Coxeter groups; he also discovered the remarkable additivity property:
\begin{equation*}
\la{addit0}
\kz_k\circ\kz_{k'}=\kz_{k+k'}
\end{equation*}
which holds for all integral $ k $ and $ k' $.
However, the key arguments in \cite{O1} involve continuous deformations in parameter $k$ and work only under
the assumption that $\dim\,\H_k=|W|$, which still remains a conjecture
for some exceptional groups in the complex case (see \cite{BMR}).
We will derive basic properties of $ \kz_k\, $, including the above additivity, from the properties of the category $\O_k $;
thus, we will give a complete case-free proof of Opdam's results (see Theorem~\ref{zero} and Corollary~\ref{Opcon}).

The link to quasi-invariants is explained by Proposition~\ref{sq},
which says that, for any $\,\tau \in \W \,$, the $H_k$-module
$\,\QQ_k(\tau)\,$ is isomorphic to the so-called standard module
$\, M_k(\tau')$ taken, however, with a twist\footnote{This result corrects an error in \cite{BEG} (cf. Remark~\ref{mistake} in Section~\ref{MQ}).}: $\,\tau' =\kz_{-k}(\tau)$.
We would also like to draw reader's attention to formula \eqref{shf00},
which gives an intrinsic description of the module $ \QQ_k(\tau) $ and should be taken,
perhaps, as a conceptual definition of quasi-invariants (see Remark~\ref{Rdefq}).

In Section~\ref{MQ}, we will use the above description of quasi-invariants
to prove Theorem~\ref{chev} and find a decomposition of $ Q_k $ as a module over the spherical algebra
$\,U_k= \e\,H_k\e$. In addition, we compute the Poincar\'e series of $Q_k$,
generalizing the earlier results of \cite{FV1}, \cite{EG1} and \cite{BEG} to the complex case.
As an application, we give a simple proof of a theorem of Opdam on symmetries of
fake degrees of complex reflection groups.

The paper ends with an Appendix, which links our results to the original setting of \cite{CV90}. For a general complex reflection group $W$ and $W$-invariant integral multiplicities $\, k = \{k_{H, \,i}\}\,$, we define the {\it Baker-Akhiezer function} $\,\psi(\lambda,\,x)\,$ and establish its basic properties. Although this function is not used
in the main body of the paper, it is certainly worth studying.

\subsection*{Acknowledgement}
We are especially indebted to Eric Opdam who has carefully read a preliminary version of this paper
and made many useful suggestions. Thanks to his effort, many proofs are now considerably shorter,
and the whole exposition has greatly improved.
We are also very grateful to Toby Stafford and Michel Van den Bergh for sharing with us their private notes \cite{SvdB}. The main results of this paper -- Theorems~\ref{chev} and~\ref{ma} -- have been independently established in \cite{SvdB} for the complex reflection groups of type $ G(m,p,N) $. In addition, we would like to thank C.~Dunkl, P.~Etingof, V.~Ginzburg, I.~Gordon, R.~Rouquier and A.~Veselov for interesting discussions and comments.

The first author is grateful to the London Mathematical Society for a travel grant and the Mathematics Department of Leeds University
for its hospitality during his visit in March 2007.

\section{Definition of Quasi-invariants}
\la{sec1}
\subsection{Complex reflection groups}
\la{1.1}
Let $ V $ be a finite-dimensional vector space over $ \c$, and let
$ W $ be a finite subgroup of $ \GL(V) $ generated by complex
reflections. We recall that an element $ s \in \GL(V)  $ is
a {\it complex reflection} if it
acts as
identity on some hyperplane $ H_s $ in $ V $.
Since $W$ is finite, there is a positive definite
Hermitian form $\,(\,\cdot\,,\,\cdot\,)\,$ on $ V $, which is invariant under
the action of $W$. We fix such a form,
once and for all, and regard $W$ as a subgroup of the corresponding unitary
group $ U(V) $. We assume that $\,(\,\cdot\,,\,\cdot\,) \,$ is antilinear in its
first argument and linear in the second: if $ x \in V $, we write
$ x^* \in V^* $ for the linear form: $\, V \to \c \,$,
$\,v \mapsto (x,\,v)\,$. The assignment $\, x \mapsto x^* \,$ defines then an
antilinear isomorphism $\,V \stackrel{\sim}{\to} V^*$, which extends to an
antilinear isomorphism of the symmetric algebras
$\c[V^*] $ and $ \c[V] $.

Let $ \A $ denote the set $\,\{H_s\}\,$ of reflection hyperplanes of $ W $,
corresponding to the reflections $ s \in W $. The group $ W $ acts on $ \A $
by permutations, and we write $\A/W$ for the set of orbits of $W$ in $ \A $.
If $ H \in \A $, the (pointwise) stabilizer of $H$ in $W$
is a cyclic subgroup $ W_H \subseteq W $ of order $ n_H $, which
depends only on the orbit $ C_H \in \A/W $ of $H$ in $\A $.
We fix a vector $ v_H \in V $, normal to $ H $
with respect to $(\,\cdot\,,\,\cdot\,)$, and a covector
$ \alpha_H \in V^*$, annihilating $ H $  in $ V^*$.
With above identification, we may (and often will) assume that
$ \alpha_H = v_H^* $.

Now, we write $ \det: W \to \c^\times $ for the character of $W$
obtained by restricting the determinant character of $\GL(V)$.
Then, under the natural action of $W$, the elements
\begin{equation}
\la{delta} \delta := \prod_{H\in \A}\alpha_H \in \c[V]\ ,
\qquad \delta^*:=\prod_{H\in \A}v_H \in \c[V^*]\ ,
\end{equation}
transform as relative invariants with characters $\det^{-1}$ and
$\det$, respectively.
For each $ H \in \A $, the characters of $ W_H $ form a cyclic
group of order $ n_H $ generated by $ \det|_{W_H} $. We write
\begin{equation}
\la{idem}
\e_{H, i} := \frac{1}{n_H}\,\sum_{w \in W_H} (\det\,w)^{-i} w
\end{equation}
for the corresponding idempotents in the group algebra
$\, \c W_H \subseteq \c W \,$.

More generally, for any orbit $\,C \in \A/W\,$, we define
\begin{equation}
\la{delc}
\delta_C := \prod_{H\in C}\alpha_H \in \c[V]\ ,\qquad
\delta_C^* := \prod_{H\in
C}v_H \in \c[V^*]\  .
\end{equation}
These are also relative invariants of $W$, whose characters will be denoted
by $\det_C^{-1}$ and $\det_C$. Note that
$\det_C(s)=\det(s) $ for any reflection $ s \in W $ with $\,H_s\in C$, while
$\det_C(s)=1 $ for all other reflections.
The whole group of characters of $ W $ is generated
by $\det_C$ for various $C \in \A/W$.

Throughout the paper, we will use the following conventions.
%

 \vspace{.8ex}

{\bf 1.} A $W$-invariant function on $ \A $ and the
corresponding function on $ \A/W $ will be denoted by the same
symbol: for example, if $ C $ is the orbit of $ H $ in $ \A $,
we will often write $\, n_C \,$, $\, k_{C}\,$, $\,\ldots\,$
instead of $\, n_H \,$, $\,k_{H}\,$, etc.

{\bf 2.} The index set $\, \{0,\,1,\,2,\,\ldots,\, n_H -1 \}\,$
will be identified with $ \Z/n_H\Z $: thus we will often assume
$\,\{\e_{H,i}\}\,$,$\,\{k_{C,i}\}\,$, $\,\ldots\,$ to be indexed
by all integers with understanding that $\, \e_{H,i} =
\e_{H,\,i+n_H} \,$, $\,k_{C,i} = k_{C,\,i+n_C}\,$, etc.
%

\subsection{Quasi-invariants}
For each $ C \in \A/W $, we fix a sequence of non-negative integers
$\,k_C =\{k_{C,i}\}_{i= 0}^{n_C-1} $, with
$k_{C,0}=0$, and let $ k := \{k_C\}_{C \in \A/W} $ .
Following our convention, we will think of $\, k \,$ as a
collection of {\it multiplicities} $\, \{k_{H,i}\} \,$
assigned to the reflection hyperplanes of $ W $.
\begin{defi}
\la{def1} A {\it $k$-quasi-invariant} of $W$ is a polynomial $\, f
\in \c[V] \, $ satisfying
\begin{equation}
\label{qc} \e_{H,-i} (f) \equiv 0\ \mbox{mod}\,\langle\alpha_H
\rangle^{n_H k_{H,i}}
\end{equation}
for all $\, H\in\A\,$ and $\, i = 0, 1, \ldots,\, n_H-1 \,$.
Here $ \langle\alpha_H\rangle $ stands for the principal ideal of
$ \c[V] $
generated by $ \alpha_H $. (Note that  \eqref{qc} holds
automatically for $ i = 0 $, as we assumed $ k_{H,0} = 0 $ for all
$ H \in \A $.)
\end{defi}
We write $ Q_k(W) $ for the set of all $k$-quasi-invariants
of $ W $: clearly, this is a linear subspace of $ \c[V]$.

\begin{example}[``The Coxeter case'']
\la{Cox} Let $W$ be a finite Coxeter group. Then
each $\, W_H \,$ is generated by a real reflection $ s_H $ of
order $\, n_H = 2\,$, and the corresponding idempotents
\eqref{idem} are given by $\,\e_{H,0} = (1 + s_H)/2\,$ and
$\,\e_{H,1} = (1 - s_H)/2\,$. As $\, k_{H, 0} = 0 \,$,
we have only one (nontrivial) condition \eqref{qc} for each $H
\in \A $, defining quasi-invariants: namely, $\, s_H(f) \equiv f
\,\mbox{mod}\,\langle\alpha_H\rangle^{2 k_H}$, with $ k_H =
k_{H,1} $. This agrees with the original definition of quasi-invariants
for the Coxeter groups (cf. \cite{FV}).
\end{example}
\begin{example}[``The one-dimensional case'']
\la{1dim}
Fix an integer $\, n \ge 2 $, and let $\,W \,$ be $\, \Z/n\Z \,$
acting on $V=\c$ by multiplication by the $n$-th roots of unity.
In this case, we have
only one reflection ``hyperplane'' -- the origin -- with
multiplicities $ k = \{k_0=0,\,
k_1,\,\ldots\,,\, k_{n-1}\} $. Identifying $ \c[V] \cong \c[x] $,
it is easy to see that
\begin{equation}
\la{exq}
Q_k(W) = \bigoplus_{i=0}^{n-1}\,x^{n k_i+i}\, \c[x^n]\ .
\end{equation}
Observe that the first summand in \eqref{exq} (with $i=0$) is
$\,\c[x^n]=\c[V]^W$, the ring of invariants of $ W $ in $ \c[V] $.
Observe also that $ Q_k $ contains all sufficiently large powers of
$\, x \,$ and hence the ideal $\, \langle x \rangle^N \subset
\c[V] $ for some $ N \gg 0 $.
In general, $Q_k$ is not a ring: it is not closed under
multiplication in $ \c[V]$. However, we can define $\,A_k := \{p
\in \c[x]\,:\, p\, Q_k \subseteq Q_k\}\,$, which is obviously a
graded subring of $ \c[V] $, $\, Q_k $ being a graded $A_k$-module.
It is easy to see that $ A_k $ also consists of quasi-invariants of $ W $,
corresponding to different multiplicities (cf. Lemma~\ref{alg} below).
Letting $ X_k := \Spec(A_k)\,$, we note that $ X_k $ is a
rational cuspidal curve, with a unique singular point ``at the
origin.'' The space $\, Q_k \,$ can be thought of geometrically, as
the space of sections of a rank one torsion-free coherent sheaf on
$ X_k $. As a $\c[V]^W$-module, $ Q_k $ is freely generated by the
monomials $\{x^{nk_i+i}\}$,\, $i=0,\dots, n-1$.
\end{example}
\subsection{Elementary properties of quasi-invariants}
\la{prop} We now describe some properties of quasi-invariants,
which follow easily from Definition~\ref{def1}. First, as in
Example~\ref{1dim} above, we fix  $ k = \{ k_{H, i}\} $ and set
\begin{equation}
\la{acc}
A_k:= \{p \in\mathbb C[V]\ :\ p \,Q_k \subseteq Q_k\}\ .
\end{equation}
The following lemma is a generalization of \cite{BEG}, Lemma~6.3.
\begin{lemma}\hfill

\la{alg}

$\mathsf{(i)}\ $ $ A_k = Q_{k'}(W) \,$ for some $ k' =
\{k_{H,i}'\}$. In particular, both $ Q_k $ and $ A_k $ contain $
\c[V]^W $ and are stable under the action of $ W $.

$\mathsf{(ii)}\ $ $ A_k $ is a finitely generated graded subalgebra
of $ \c[V] $, and $Q_k$ is a finitely generated graded module
over $A_k$ of rank $1$.

$\mathsf{(iii)}\ $ The field of fractions of  $ A_k $ is $\, \c(V)
\,$, and the integral closure of $\,A_k \,$ in $\, \c(V) \,$ is
$\,\c[V]\,$.
\end{lemma}

\begin{proof}
For a polynomial $ f \in \c[V] $, we define its normal expansion
along a hyperplane $ H \in \A $ by
$$
f(x + t v_H) = \sum_{s \ge 0} c_{H,\,s}(x)\,t^s \ ,\quad x \in H\ .
$$
It is then easy to see that $ f $ satisfies \eqref{qc} if and only
if $\, c_{H,\,s}(x) = 0 \,$ for all $\, s \in\Z_+\!\setminus S
\,$, where
$$
S = \bigcup_{i=0}^{n_H-1}\{i+n_H k_{H,i}+n_H\Z_{+}\}\ .
$$
Now, letting $\,R := \{r \in \Z \ :\ r + S \subseteq S\}\,$, we
observe that $\, p \in A_k \,$ if and only if, for each $ H \in \A
$, the normal expansion of $ p $ along $ H $ contains no terms $ t^r
$ with $\, r \not\in R \,$. To prove $ \mathsf{(i)} $ it suffices to
note that $ R $ can be written in the same form as $\, S \,$, maybe
with different $ k$'s. Indeed, $\,S\subset \Z\,$ can be
characterized by the property that it is invariant under translation
by $n_H$ and contains all integers $\,s\gg 0$. Clearly, $R$ has
the same property and, therefore, a similar description.

To prove $ \mathsf{(ii)} $ and $ \mathsf{(iii)} $, we can argue
as in \cite{BEG}, Lemma~6.3. Since $\,\c[V]^W \subseteq A_k \subseteq
\c[V]\,$, the Hilbert-Noether Lemma implies that $ A_k $ is a
finitely generated algebra, and $\c[V] $ is a finite module over $
A_k $. Being a submodule of $\c[V] $, $ Q_k $ is then also
finite over $ A_k $. Now, both $ A_k $ and $ Q_k $ contain the
ideal of $ \c[V] $ generated by a power of $
\delta \in \c[V] $. Hence, $ A_k $ and $ \c[V] $ have the same
field of fractions, namely $\,\c(V)\,$, and the integral closure
of $ A_k $ in $ \c(V) $ is $ \c[V] $. This also implies that
$\,\dim_{\c(V)}[Q_k \otimes_{A_k} \c(V)] = 1 \,$, and thus $ Q_k $
is a rank $1$ module over $ A_k$.
\end{proof}

It is convenient to state some properties of
quasi-invariants in geometric terms. To this end, we write $ \,
X_k = \Spec(A_k) $ and let $\O_x = \O_x(X_k)$ denote the local ring
of $ X_k $ at a point $ x \in X_k $. This local ring can be
identified with a subring of $ \c(V) $ by localizing the algebra
embedding $ A_k \into \c[V] $. To the module $Q_k$ we can then
associate a torsion-free coherent sheaf on $X_k$, with fibres $
(Q_k)_x = Q_k \otimes_{\O} \O_x $. Our definition of
quasi-invariants generalizes to this local setting if we require
\eqref{qc} to hold for the stabilizer $ W_x $ of $ x $ under the
natural action of $W$ on $ X_k $. This makes sense, since by a
theorem of Steinberg \cite{St}, $ W_x $ is also generated by
complex reflections.
\begin{lemma}[cf. \cite{BEG}, Lemma~7.3] \la{inj}
Let $\, \a^n := \Spec\ \c[V] $.

$\mathsf{(i)}\ $ $ X_k $ is an irreducible affine variety, with
normalization $ \tilde{X}_k = \a^n $.

$\mathsf{(ii)}\ $ The normalization map $\,\pi_k :\ \a^n \to X_k \,$
is bijective.

$\mathsf{(iii)}\ $ If we identify the (closed) points of $ X_k $
and $ \a^n $ via $ \pi_k $, then for each $ x \in \a^n $, $\, (Q_k)_x $ is
the space of $k$-quasi-invariants in $\c(V)$ with respect to the
subgroup $W_x\subseteq W$.
\end{lemma}
\begin{proof} The proof given in \cite{BEG} in the case of Coxeter groups (see \cite{BEG},
Lemma~7.3) works, {\it mutatis mutandis},
for all complex reflection groups. We leave this  as a (trivial) exercise to the reader.
\end{proof}

\section{Quasi-invariants and Cherednik Algebras}
\la{sec3}
\subsection{The rational Cherednik algebra}
\la{cha}
We begin by reviewing the definition of Cherednik
algebras associated to a complex reflection group.
For more details and proofs, we refer the reader to \cite{DO} and \cite{GGOR}. In this section,
unless stated otherwise, the multiplicities $\, k_{C,i}\,$ are
assumed to be arbitrary complex numbers.

We set $\, \vreg := V \setminus \bigcup_{H \in \A} H\,$ and let $ \c[\vreg] $ and $ \D(\vreg) $
denote the rings of regular functions and regular differential operators on $ \vreg $, respectively. The
action of $ W $ on $ V $ restricts to $ \vreg $, so $ W $
acts naturally on $ \c[\vreg] $ and $ \D(\vreg) $ by algebra automorphisms. We form the crossed products
$ \c[\vreg]*W $ and $ \D(\vreg)*W $ and denote $ \D W :=  \D(\vreg)*W$. As an algebra,
$\D W$ is generated by its two subalgebras $ \c W $ and $ \D[\vreg] $, and hence, by the elements
of $ W $, $\, \c[\vreg] \,$ and the derivations $\,
\partial_\xi \,$, $\, \xi \in V $.

Following \cite{DO}, we now define the {\it Dunkl operators}
$\,T_{\xi} \in \D W \,$ by
\begin{equation}
\label{du} T_\xi :=
\partial_\xi-\sum_{H\in \A}
\frac{\alpha_H(\xi)}{\alpha_H}\sum_{i=0}^{n_H-1}n_H k_{H,i}
\e_{H,i}\ , \quad \xi \in V\ .
\end{equation}
Note that the operators \eqref{du} depend on $\, k = \{k_{H,i}\} \,$, and
we sometimes write $ T_{\xi, k} $ to emphasize this dependence.
The basic properties of Dunkl operators are gathered in the following lemma.
\begin{lemma}[see \cite{D, DO}]\la{duprop}
For all $\,\xi, \eta \in V\,$ and $ w \in W $, we have

$\mathsf{(i)}$\ {\rm commutativity:}\ $\, T_{\xi,k}\,T_{\eta, k} -
T_{\eta,k}\,T_{\xi, k} = 0 \,$,

$\mathsf{(ii)}$\ {\rm $W$-equivariance:}\ $\,w\,T_\xi =
T_{w(\xi)}\,w\,$,

$\mathsf{(iii)}$\ {\rm homogeneity:}\ $ \,T_\xi \,$ is a
homogeneous operator  of degree $ -1 $
 with respect to the natural (differential) grading on $ \D W $.
\end{lemma}
Properties $\mathsf{(ii)}$ and $\mathsf{(iii)}$ of Lemma~\ref{duprop}
follow easily from the definition of Dunkl operators. On
the other hand, the commutativity $\mathsf{(i)}$ is far from being
obvious: it was first proved in \cite{D} in the Coxeter case, and
then in \cite{DO} in full generality ({\it loc. cit.}, Theorem~2.12).

In view of Lemma~\ref{duprop}, the assignment $\,\xi \mapsto T_\xi\,$
extends to an injective algebra homomorphism
\begin{equation}
\la{hom}
\c[V^*] \hookrightarrow \D W \ ,\quad p \mapsto T_p \ .
\end{equation}
Identifying $ \c[V^*] $ with its image in $ \D W $ under
\eqref{hom}, we now define the {\it rational Cherednik algebra}
$H_k=H_k(W)$ as the subalgebra of $\D W$ generated by $ \c[V] $,
$\,\c[V^*] $ and $ \c W$.

The Cherednik algebras can be also defined directly, in terms of
generators and relations, see \cite{EG, BEG, GGOR}. To be precise,
$H_k$  is generated by the elements $ x \in V^*$, $\xi \in V $ and $
w \in W$ subject to the following relations
\begin{eqnarray*}\la{chrel}
&&[x,\,x'] =0\ ,\quad [\xi,\,\xi']=0 \ ,\quad
w\,x\,w^{-1} =w(x)\ ,\quad w\,\xi\,w^{-1}=w(\xi)\ ,\nonumber \\*[1ex]
&& [\xi,\,x] = \langle \xi,\,x \rangle + \sum_{H\in \A}\frac{\langle \alpha_H, \xi\rangle\,
\langle x, v_H\rangle}{\langle \alpha_H, v_H\rangle}\sum_{i=0}^{n_H-1}n_H(k_{H,i}-k_{H,i+1})\,\e_{H,i}\ .
\end{eqnarray*}
The family $ \{H_k\} $ can be viewed as a deformation (in fact, the
universal deformation) of the crossed product $\, H_0 = \D(V)*W \,$
(see \cite{EG}, Theorem~2.16). The embedding of $\, H_k \into \D W
\,$ is given by $\,w \mapsto w\,$, $\,x \mapsto x\,$ and $\,\xi
\mapsto T_{\xi}\,$ and referred to as the {\it Dunkl representation}
of $ H_k $. The existence of such a representation implies the
Poincar\'{e}--Birkhoff--Witt (PBW) property for $ H_k $, which says
that the multiplication map
\begin{equation}\la{pbw}
\c [V] \otimes \c W \otimes \c[V^*]
\stackrel{\sim}{\to} H_k\
\end{equation}
is an isomorphism of vector spaces (see \cite{EG})\footnote{The PBW property is
proven in \cite{EG} for a more general class of symplectic reflection algebras.}.

The algebra $ \D W = \D(\vreg) * W $ carries two natural
filtrations: one is defined by taking $\,\deg(x) = \deg(\xi) =1 $,
$\,\deg(w)=0$, and the other is defined by $\,\deg(x) =0 $,
$\,\deg(\xi) = 1 $, $\,\deg(w) = 0$ for all $x\in V^*$, $\xi\in V$
and $w\in W$. We refer to the first filtration as {\it standard} and
to the second as {\it differential}. Through the Dunkl
representation, these two filtrations induce filtrations on $ H_k $
for all $ k $. It is easy to see that the associated graded rings
$\,\grd\, H_k\,$ are isomorphic to $\,\c[V\times V^*]*W$ in both
cases; in particular, they are independent of $ k $.

Note that $ \{1,\,\delta,\, \delta^2,\,\ldots\} $, with $ \delta
$ defined in \eqref{delta}, is a localizing (Ore) subset in $ H_k $:
we write $\,\Hreg := H_k[\delta^{-1}] \,$ for the corresponding
localization. Since $ \delta $ is a unit in $ \D W $,
the Dunkl embedding $\, H_k \into \D W \,$ induces a canonical
map $\,\Hreg \to \D W \,$.
\begin{prop}[see \cite{EG}, Prop.~4.5; \cite{GGOR}, Theorem~5.6]
\la{Loc} The map  $\,\Hreg \to \D W \,$  is an isomorphism of algebras.
\end{prop}
Despite its modest appearance, Proposition~\ref{Loc} plays an important
r\^ole in representation theory of Cherednik algebras. In particular,
it justifies our notation $ \Hreg $ for the localization of
$ H_k $ (as $ \Hreg $ is indeed independent of $k$).

Next, we introduce the {\it spherical subalgebra} $ U_k $ of $ H_k
$: by definition, $\, U_k := \e\, H_k \,\e \,$, where $\, \e :=
|W|^{-1} \sum_{w \in W } w \,$ is the symmetrizing idempotent in $
\c W \subset H_k $.
For $ k = 0 $, we have $\, U_0 = \e [\D(V)
* W] \e \cong \D(V)^W \,$; thus, the family $ \{U_k\}$ is a
deformation (in fact, the universal deformation) of the ring of
invariant differential operators on $ V $. The standard and differential
filtrations on $H_k$ induce filtrations on $U_k$, and we have
$\grd\, U_k \cong \c[V\times V^*]^W$ in both cases.

The relation between $ H_k $ and $ U_k $ depends drastically
on multiplicity values. In the present paper, we will be mostly concerned
with integral $k$'s, in which case we have the following result.
\begin{theorem}
\la{morita}
If $ k $ is integral, i.~e. $\, k_{C, i} \in \Z \,$ for
all $ C \in \A /W $, then  $ H_k $ and $ U_k $  are simple
algebras, Morita equivalent to each other.
\end{theorem}
\begin{proof}
There is a natural functor relating the module categories of $ H_k
$ and $ U_k \,$:
\begin{equation}
\la{Meq}
\Mod(H_k) \to \Mod(U_k) \ , \quad M \mapsto \e M \ ,
\end{equation}
where $\, \e M := \e H_k \otimes_{H_k} M \,$. By standard Morita
theory (see, e.~g, \cite{MR}, Prop.~3.5.6), this functor is an
equivalence if (and only if) $\, H_k \e \, H_k = H_k $. The last
condition holds automatically if $ H_k $ is simple. So one needs
only to prove the simplicity of $ H_k $. In the Coxeter case, this
is the result of \cite{BEG}, Theorem~3.1. In general, the
simplicity of $ H_k $ can be deduced from the semi-simplicity of the
category $\, \O_{H_k} $ for integral $k$'s, which, in turn, follows from
general results of \cite{GGOR}. We discuss this in detail in Section
\ref{o} (see Theorem~\ref{ss} below).
\end{proof}

The restriction of the Dunkl representation
$\,H_k \into \D W \,$ to $ \e H_k \e \subset H_k $ yields an
embedding $\,U_k \into \e \,\D W \e \,$, which is a homomorphism of
unital algebras. If we combine this
with (the inverse of) the isomorphism $\, \D(\vreg)^W
\stackrel{\sim}{\to} \e \,\D W \e \,$, $\, D \mapsto \e D \e = \e
D = D \e \,$, we get an algebra map
\begin{equation}
\la{HC}
\Res:\,U_k \into \D(\vreg)^W \ ,
\end{equation}
representing $ U_k $ by invariant differential operators on $ \vreg
$ (cf. \cite{H}). We will refer to \eqref{HC} as the {\it Dunkl
representation} for the spherical subalgebra $ U_k $.

\subsection{$\c W$-valued quasi-invariants}
\la{nqi}

The algebra $ \D W $ can be viewed as a ring of $W$-equivariant
differential operators on $ \vreg $, and as such it acts naturally
on the space of $ \c W$-valued functions. More precisely, using
the canonical inclusion $\, \c[\vreg] \otimes \c W \into \D W $,
we can identify $\, \c[\vreg] \otimes \c W $ with the cyclic $\D
W$-module $ \D W/J $, where $ J $ is the left ideal of $ \D W $
generated by  $\,
\partial_\xi \in \D W $, $\, \xi \in V \,$. Explicitly, in terms
of generators, $\D W$ acts on $ \c[\vreg] \otimes \c W $ by
\begin{eqnarray}
&& g(f\otimes u) =  gf\otimes u\,,\quad g\in\c[\vreg]\ , \nonumber\\
\la{acts}
&& \partial_\xi(f\otimes u) = \partial_\xi f\otimes u\,,\quad \xi\in V\ ,\\
\nonumber
&& w(f\otimes u)=f^w\otimes wu\,,\quad w\in W\ .
\end{eqnarray}

Now, the restriction of scalars via the Dunkl representation
$\,H_k(W) \into \D W $ makes $ \c[\vreg]\otimes \c W $ an $ H_k(W)
$-module. We will call the corresponding action of $ H_k $ the {\it
differential action}. It turns out that, in the case of integral
$k$'s, the differential action of $ H_k $ is intimately related to
quasi-invariants $Q_k=Q_k(W)$.
\begin{theorem}
\la{Qfat} If $k$ is integral, then  $ \c[\vreg]\otimes \c W  $
contains a unique $H_k$-submodule $ \QQ_k = \QQ_k(W) $, such that
$ \QQ_k $ is finite over $ \c[V] \subset H_k $ and
\begin{equation}
\la{proj} \e \,\QQ_k = \e (\,Q_k\otimes 1) \quad \mbox{in}\quad
\c[\vreg]\otimes \c W  \ .
\end{equation}
\end{theorem}

We prove Theorem~\ref{Qfat} in several steps. First, we
construct $ \QQ_k $ as a subspace of $\, \c[\vreg] \otimes \c W $
and verify \eqref{proj}. Then we show that $ \QQ_k $ is stable
under the differential action of $ H_k $, and finally we prove its
uniqueness.

Besides the diagonal action \eqref{acts}, we will use another action
of $W$ on $\, \c[\vreg] \otimes \c W $, which is trivial on the
first factor: i.~e., $\, f
\otimes s \mapsto f \otimes w s\,$, where $ w \in W $ and $\, f
\otimes s \in \c[\vreg] \otimes \c W \,$. We denote this action by
$\, 1 \otimes w \,$.

Now, we define $ \QQ_k $ to be the subspace of $\,\c [\vreg] \otimes
\c W $ spanned by the elements $ \varphi $ satisfying
\begin{equation}
\label{qcc} (1 \otimes \e_{H,i})\,\varphi \equiv 0 \ \mbox{mod}\
\langle\alpha_H\rangle^{n_H k_{H,i}} \otimes \c W \ ,
\end{equation}
for all  $\, H \in \A\,$ and $\,i = 0, 1, \ldots, n_H-1 \,$.
Here, as in Definition~\ref{qc}, $\,\langle \alpha_H \rangle\,$
stands for the ideal of $ \c[V] $ generated by  $ \alpha_H $.

It is immediate from \eqref{qcc} that $ \QQ_k \subseteq \c[V]
\otimes \c W $, and $ \QQ_k $ is closed in $ \c[V] \otimes \c W $
under the natural action of $ \c[V] $. Hence, as $W$ is finite and
$ \c[V] $ is Noetherian, $ \QQ_k $ is a finitely generated $
\c[V]$-module.

\begin{lemma}
\la{pr} $ \QQ_k $ satisfies \eqref{proj}.
\end{lemma}
\begin{proof}
We need to show that $\,\e (f \otimes 1) \in \e\,\QQ_k $ if and
only if $\,f\in Q_k\,$.
First, for any $\,f\in\c[V]\,$ and $\,s\in W\,$, we compute
\begin{equation*}
(1 \otimes s)[\e (f \otimes 1)] =\frac{1}{|W|}\,\sum_{w \in W} \,
f^w \otimes sw
=\frac{1}{|W|}\sum_{w\in W}f^{s^{-1}w} \otimes w\,.
\end{equation*}
Now, multiplying this by appropriate characters and summing up over
all $\,s \in W_H \,$, we get
\begin{equation*}
(1 \otimes \e_{H,i})[\e(f \otimes 1)] = \frac{1}{|W|}\, \sum_{w\in
W}\, \e_{H,-i}(f^{w})\otimes w\ .
\end{equation*}
It follows from \eqref{qcc} that $\, \e (f \otimes 1) \in \e \QQ_k \,$ if
and only if $\,f^w \in Q_k $ for all $w \in W $. The latter is equivalent
to $\, f\in Q_k$, since $ Q_k $ is $W$-stable.
\end{proof}
\begin{lemma}
\la{inv} $ \QQ_k $ is stable under the differential action of $ H_k $.
\end{lemma}
\begin{proof}
As already mentioned above, $ \QQ_k $ is closed under the action of
$\,\c[V] \subset H_k\,$. To see that $ \QQ_k $ is stable under the
diagonal action of $ W $, we observe that
\begin{equation*}
w \, (1 \otimes \e_{H,i}) = w \otimes w\e_{H,i} = (1 \otimes
\e_{wH,i})\, w \,
\end{equation*}
as endomorphisms of $\c [\vreg] \otimes \c W $.
Since \eqref{qcc}
hold for each $ H \in \A $ and $ k_{H,i}$'s depend only on the
orbit of $ H $ in $ \A $, we have $\, w\,\QQ_k \subseteq \QQ_k
$ for all $w\in W$.

Thus, we need only to check that $ \QQ_k $ is preserved by the
Dunkl operators \eqref{du}. For each $\, H\in \A \,$, let $\,
\QQ^H_k  $ denote the subspace of $\,\c[\vreg]\otimes \c W \,$
spanned by all $ \varphi$'s satisfying \eqref{qcc} only for the
given $H$. Clearly $\, \QQ_k = \bigcap_{H\in \A} \QQ^H_k$, so
it suffices to show that
\begin{equation}
\la{tin}
T_\xi(\QQ_k) \subseteq \QQ^H_k \quad \text{for all}\ H\in \A\ .
\end{equation}
Writing $\,T_\xi = T_0 + T_1\,$ with
\begin{eqnarray*}
&& T_0 := \partial_\xi-\frac{\alpha_H(\xi)}{\alpha_H}\sum_{i=0}^{n_H-1} n_H k_{H,i} \e_{H,i}\,,\\
&& T_1 := \sum_{H'\ne H}\frac{\alpha_{H'}(\xi)}{\alpha_{H'}}\sum_{i=0}^{n_{H'}-1}n_{H'} k_{H',i}
\e_{H',i}\ ,
\end{eqnarray*}
we will verify \eqref{tin} separately for $ T_0 $ and $ T_1 $.

Since $\QQ_k $ is $W$-stable, $\, \e_{H',i}(\QQ_k) \subseteq \QQ_k
\subseteq \QQ_k^H$. Next, $\,\alpha_{H'}^{-1} \in \c [\vreg] \,$ is
regular along $H$, therefore, $ \alpha_{H'}^{-1} \,\QQ_k^{H}
\subseteq \QQ_k^{H} $. Combining these two facts together, we get
$\,\alpha_{H'}^{-1} \e_{H',i}(\QQ_k) \subseteq \QQ_k^H \,$, and
hence $\,T_1(\QQ_k) \subseteq \QQ^H_k $.

It remains to show that $\,T_0(\QQ_k) \subseteq \QQ^H_k \,$. In
fact, we have $\QQ_k\subseteq \QQ^H_k$, so it suffices to show that
$\,T_0(\QQ^H_k) \subseteq \QQ^H_k \,$. Note that the definition of
both $\QQ^H_k$ and $T_0$ involve only one hyperplane $H$ and the
group $W_H$, so the statement can be checked in dimension one, in
which case it is straightforward, see Example \ref{exx} below.
\end{proof}
\begin{lemma}
\la{un} If $ k $ is integral, there exists at most one $H_k$-submodule
$\, \QQ_k \subset
\c[\vreg] \otimes \c W \,$, satisfying \eqref{proj}.
\end{lemma}
\begin{proof}
Suppose that $ \QQ_k $ and $ \QQ_k' $ are two such submodules.
Replacing one of them by their sum, we may assume that $\QQ_k
\subset \QQ_k'$, with $e\QQ_k=e\QQ_k'$. Setting $\,M:=\QQ_k'/\QQ_k$,
we get $eM=0$. This forces $\, M = 0\,$, since \eqref{Meq} is a
fully faithful functor by Theorem~\ref{morita}. Thus $\, \QQ_k' =
\QQ_k$, as required.
\end{proof}

Lemmas~\ref{pr}, \ref{inv} and \ref{un} combined together imply
Theorem~\ref{Qfat}. As a simple consequence of this theorem, we get
\begin{cor}
\la{sp} $\,Q_k \,$ is stable under the action of $\, U_k \,$ on $
\c[\vreg] $ via the Dunkl representation \eqref{HC}.
Thus $ Q_k $ is a $ U_k$-module, with $ U_k $ acting on
$ Q_k $ by invariant differential operators.
\end{cor}
\begin{proof}
Theorem \ref{Qfat} implies that $\e H_k \e(\e\QQ_k)\subseteq
\e\QQ_k$. Recall that for every element $\e L \e\in\e H_k\e$ we have
$\e L \e=\e\,\Res\,L$, by the definition of the map \eqref{HC}. As a
result,
$$
\e\,(\Res\,L[Q_k]\otimes 1) = \e\,\Res\,L\,[Q_k\otimes 1] = (\e\,
L\, \e)[\QQ_k] \subseteq \e \QQ_k = \e (Q_k\otimes 1)\,.
$$
It follows that $\, (\Res\,L)[Q_k] \subseteq Q_k \,$, since $\,
\e\,(f \otimes 1) = 0\,$ in $\c[\vreg]\otimes \c W$ forces $\,f = 0
\,$.
\end{proof}

\begin{example}\la{exx}
We illustrate Theorem \ref{Qfat} in the one-dimensional case.
Let $\,W=\Z/n\Z \,$ and $\, k = (k_0, \dots, k_{n-1})$ be as in
Example \ref{1dim}. Then
\begin{equation}\la{nc11}
\QQ_k=\bigoplus_{i=0}^{n-1} x^{nk_i}\c[x]\otimes\e_{i}\ ,\quad
\,\e_i=\frac{1}{n}\sum_{w\in W}(\det w)^{-i} w\,.
\end{equation}
Clearly, $ \QQ_k $ is stable under the
action of $W$ and $\c[x]$. On the other hand, if $\, k_i \in \Z\,$,
a trivial calculation shows that the Dunkl operator $T:=\partial_x-x^{-1}\sum_{i=0}^{n-1} n k_{i} \e_{i}\,$
annihilates the elements $x^{nk_i}\otimes\e_{i}$, and hence preserves $ \QQ_k $ as well.
Now, acting on $ \QQ_k $ by $\e=\e_0$ and using \eqref{exq}, we get
\begin{equation}\la{enc1}
\e\QQ_k=\bigoplus_{i=0}^{n-1}\, x^{nk_i+i}\c[x^n]\otimes\e_{i}= \bigoplus_{i=0}^{n-1}\, \e\,(x^{nk_i+i}\c[x^n]\otimes 1) =
\e\,(Q_k\otimes 1)\ ,
\end{equation}
which agrees with Theorem~\ref{Qfat}.
\end{example}

\subsection{Generalized quasi-invariants}
\la{ttau}
In our construction of quasi-invariants, the regular representation
$ \c W $ played a distinguished r\^{o}le. We now outline a
generalization, in which $\c W$ is replaced by an arbitrary
$W$-module $\tau$. For a more conceptual definition of
quasi-invariants in terms of shift functors, we refer
the reader to Section~\ref{SF} (see Remark~\ref{Rdefq}).

First, we observe that the left ideal
$\,J \,$ of $\,\D W\,$ generated by  the derivations
$\, \partial_\xi $, $\,\xi \in V \,$, is stable under
{\it right} multiplication by the elements of $ \c W \subset \D W $.
Hence $\,\D W/J \,$ is naturally a $\D W$-$\c W$-bimodule.
For any $W$-module $\,\tau \,$, we can form then the left $\D W$-module
$\,\D W/J \otimes_{\c W} \tau \cong \c[\vreg] \otimes \tau \,$.
The action of $\D W$ on $\c[\vreg] \otimes \tau$ is given by the
same formulas \eqref{acts}, with $\, w \in W\,$ acting now in
representation $\tau$, and $ H_k $ operates via its Dunkl representation.
Now, generalizing \eqref{qcc}, we
define the module $ \QQ_k(\tau) $ of {\it $\tau$-valued quasi-invariants}
as the span of all $\,\varphi \in  \c[\vreg] \otimes \tau \,$ satisfying
\begin{equation}
\label{qctau}
(1 \otimes \e_{H,i})\, \varphi \equiv 0\ \mbox{mod}\,\langle\alpha_H \rangle^{n_H k_{H,i}}
\otimes \tau\
\end{equation}
for all $\,H \in \A\,$ and $\,i = 0, 1, \ldots, n_H-1\,$.
It is convenient to write $\QQ_k(\tau)$ as the intersection of
subspaces corresponding to the reflection hyperplanes $H\in \A \,$:
\begin{equation}\label{qtauh}
\QQ_k(\tau)=\bigcap_{H\in \A}\QQ_k^H(\tau)\,,\quad
\QQ_k^H(\tau) :=\bigoplus_{i=0}^{n_H-1}
\langle\alpha_H\rangle^{n_Hk_{H,i}}\,\otimes \e_{H,i}\tau\,.
\end{equation}
The same argument as in Lemma \ref{inv} above proves the following

\begin{prop}
The space $\QQ_k(\tau)\subset \c[\vreg]\otimes\tau$ is stable
under the action of $H_k$. The subspace $\e\QQ_k(\tau)$
of $W$-invariant elements in $\QQ_k(\tau)$ is then a module over the
spherical subalgebra $\e H_k \e$.
\end{prop}
\noindent
In addition, we have
\begin{lemma}\la{inter} Let $ \QQ_k^H(\tau) $ be as in \eqref{qtauh}, and let
$\,\e_{H,0} := \frac{1}{n_H}\sum_{w\in W_H} w\,$. Then
\begin{equation}\label{eho}
\e\QQ_k(\tau)=\bigcap_{H\in \A}\e_{H,0}\,\QQ_k^H(\tau)\ . \end{equation}
\end{lemma}
\begin{proof}
First, it is clear that the right-hand side of \eqref{eho}
lies in the intersection \eqref{qtauh} and thus belongs to
$\QQ_k(\tau)$. Furthermore, it is contained in $\e_{H,0}\QQ_k^H(\tau)$ and
therefore invariant
under the action of $W_H$. Since $H$ is
arbitrary, this proves that the right-hand side of \eqref{eho} is
invariant under the whole of $W$ and hence contained in the left-hand side.
The opposite inclusion follows from
$\,\e\QQ_k(\tau)\subseteq \e_{H,0}\QQ_k(\tau)\subseteq \e_{H,0}\QQ_k^H(\tau)\,$.
\end{proof}

We can decompose each subspace $\e_{H,0}\QQ_k^H(\tau)$ in \eqref{eho} as in the
one-dimensional case (see Example~\ref{exx}, \eqref{enc1}).
To be precise, let $\c[\vreg^H]$ denote the subring of functions in
$ \c[\vreg]$ that are regular along $H$. This ring
carries a natural action of $W_H$, so we write
$\c[\vreg^H]^{W_H}$ for its subring of invariants.
With this notation, we have
\begin{equation}\label{eho1}
\e_{H,0}\QQ_k^H(\tau)=\bigoplus_{i=0}^{n_H-1}
\alpha_H^{n_Hk_{H,i}+i}\,\c[\vreg^H]^{W_H}\otimes\e_{H,i}\tau\,.
\end{equation}

We close this section with a few remarks.
\medskip

{\bf 1.} As an immediate consequence of the definition
\eqref{qctau}, we have
\begin{equation}\label{sa0}
\delta^r\c[V]\otimes\tau\subset\QQ_k(\tau)\subset\delta^{-r}\c[V]\otimes\tau\,,
\end{equation}
where $r>0$ is sufficiently large (precisely, $\,r> \max
\{n_Hk_{H,i}\}$). More generally, for integral $k, k'$, it is easy
to show that
\begin{equation}\la{sa}
\delta^r\QQ_k(\tau)\subseteq \QQ_{k'}(\tau)\subseteq
\delta^{-r}\QQ_k(\tau)\ ,
\end{equation}
where $ r \gg 0 $  depends only on the difference $k'-k$.

{\bf 2.} If $\tau$ is a direct sum of $W$-modules, say $\tau_i$,
then $\c[\vreg]\otimes\tau$ and $ \QQ_k(\tau)$ are also direct sums
of $\c[\vreg]\otimes\tau_i$ and $\QQ_k(\tau_i)$, respectively. In
particular, replacing $\tau$ by $\,\c W = \sum_{\tau \in \W }\tau
\otimes \tau^* $, we get
\begin{equation}\la{ssum}
\QQ_k =\sum_{\tau\in \W} \QQ_k(\tau)\otimes\tau^*\,.
\end{equation}
Thus, the structure of $ \QQ_k $ is determined by the modules $
\QQ_k(\tau) $ associated to irreducible representations of $ W $. We
will study these modules in detail in Section~\ref{MQ}.

{\bf 3.} As was mentioned already, on the space $\c[\vreg]\otimes\c
W$ one has yet another (left) $W$-action sending $f\otimes u$ to
$f\otimes uw^{-1}$. It is clear from the definitions, that it
commutes with the action of $\D W$ and preserves both $\QQ_k$ and
$\e \QQ_k$. Note that this action preserves each summand in
\eqref{ssum}, acting on $\tau^*$. Under \eqref{proj}, it
translates into the standard action of $W$ on $Q_k\subset \c[V]$.

\section{Differential Operators on Quasi-invariants}
\la{doq}

\subsection{Rings of differential operators}
\la{tdo}

We briefly recall the definition of differential operators in the
algebro-geometric setting (see \cite{Gr} or \cite{MR}, Chap.~15).

Let $ A $ be a commutative algebra over $ \c $, and let $ M $ be
an $A$-module. The filtered ring of (linear) differential
operators on $ M $ is defined by
$$
\D_A(M) := \bigcup_{n\ge 0}\,\D_A^n(M) \,\subseteq \,\End_{\c}(M)\
,
$$
where $\, \D^{0}_A(M) := \End_A(M) \,$ and $\, \D^n_A(M) \,$, with
$ n\ge 1 $, are given inductively:
$$
\D_A^n(M) := \{D \in \End_{\c}(M)\ |\ [\,D,\, a\,] \in
\D_A^{n-1}(M) \ \text{for all}\ a \in A\}\ .
$$
The elements of $ \D^n_A(M)\setminus \D^{n-1}_A(M) $ are called
{\it differential operators of order $ n $} on $ M $. Note that
the commutator of two operators in $\D_A^n(M) $ of orders $ n $
and $ m $ has order at most $ n+m-1 $. Hence the associated graded
ring $\,\grd\,\D_A(M) := \bigoplus_{n\ge
0}\D_A^n(M)/\D_A^{n-1}(M)\,$ is a commutative algebra.

If $ X $ is an affine variety with coordinate ring $ A =\O(X) $,
we denote $ \D_A(A) $ by $ \D(X) $ and call it the {\it ring of
differential operators on $ X $}. If $ X $ is irreducible, then
each differential operator on $ X $ has a unique extension to a
differential operator on $ \k := \c(X) $, the field of rational
functions of $ X $, and thus we can identify (see \cite{MR},
Theorem~15.5.5):
\begin{equation*}
\D(X) = \{\, D \in \D(\k)\ |\ D(f) \in \O(X) \ \text{for all} \ f
\in \O(X) \,\} \ .
\end{equation*}
Slightly more  generally, we have
\begin{lemma}[cf. \cite{BW}, Prop.~2.6]
\la{dm} Suppose that $\, M \subseteq \k $ is a (nonzero)
$A$-submodule of $\, \k $. Then
\begin{equation*}
\D_A(M) = \{\, D \in \D(\k)\ |\ D(f) \in M \ \text{\rm for all} \
f \in M \,\} \ .
\end{equation*}
\end{lemma}

We apply these concepts for $\, A = A_k \,$ and
$\, M = Q_k \,$, writing $ \D(Q_k) $ instead of $ \D_{A}(M) $ in this case.
By Lemma~\ref{alg}$\mathsf{(iii)}$, $\,X_k = \Spec(A_k) \,$ is an
irreducible variety with $ \k = \c(V) $, so, by Lemma~\ref{dm}, we
have
\begin{equation}
\la{tw} \D(Q_k) = \{D\in \D(\k)\,|\,D(f)\subseteq Q_k\ \text{\rm
for all} \ f \in Q_k \,\}\ .
\end{equation}
Note that the differential filtration on $ \D(Q_k) $ is induced from
the differential filtration on $ \D(\k) $. Thus \eqref{tw} yields
a canonical inclusion $\,\grd\,\D(Q_k) \subseteq \grd\,\D(\k)\,$,
with $\, \D^0(Q_k) = A_k \,$, see \eqref{acc}. In particular, if
$\, k = \{0\} \,$, then $ Q_k = \c[V] $ and \eqref{tw} becomes the
standard realization of $\,\D(V)\,$ as a subring of $ \D(\k)\,$.

Apart from $ Q_k$, we may also apply Lemma~\ref{dm} to $ \c[\vreg]
$, which is naturally a subalgebra of $ \k = \c(V) $. This gives
the identification
\begin{equation}
\la{tw1} \D(\vreg) = \{D\in \D(\k)\,|\,D(f)\subseteq \c[\vreg] \
\text{\rm for all} \ f \in \c[\vreg] \,\} \ .
\end{equation}
\begin{lemma}
\la{symbol} With identifications \eqref{tw} and \eqref{tw1}, we
have
$$
(\mathsf i)\ \D(Q_k)\subseteq \D(\vreg) \quad  \text{and}\quad
(\mathsf{ii})\ \grd\,\D(Q_k)\subseteq \grd\,\D(V)\ .
$$
\end{lemma}
\begin{proof}
This can be deduced from general results of \cite{SS} or
\cite{BEG} (see, e.g., \cite{BEG}, Lemma~A.1). However, for
reader's convenience, we give a shorter argument here.
First, recall that $ \, \delta^N\c[V] \subseteq Q_k \subseteq \c[V]\, $
for some $\,N \ge 1 \,$. Hence,
for any $\, D \in
\D(Q_k) \,$, we have
$$\,  D\, \delta^{N}(\c[V])\subseteq D(Q_k)\subseteq Q_k\subseteq \c[V]\,.$$
It follows that $D\delta^N \in \D(V)$ for all $D\in \D(Q_k)$ and $\,\D(Q_k)\subseteq
\D(V)\,\delta^{-N}$ proving the first claim of the lemma. The
last inclusion also implies that $\, \grd\,\D(Q_k)\subseteq
\delta^{-N}\grd\,\D(V)\,$. Since $ \grd\,\D(Q_k) $ is closed under
multiplication, this is possible only if $\,\grd\,\D(Q_k)\subseteq
\grd\,\D(V) $, which is the second claim of the lemma.
\end{proof}
\subsection{Invariant differential operators}

Recall that, by Lemma~\ref{alg}, $\, Q_k \,$ is stable under the
action of $W$ on $ \c[\vreg] $. Hence $ W $ acts naturally on $
\D(Q_k) $, and this action is compatible with the inclusion of
Lemma~\ref{symbol}($\mathsf i$). It follows that $\, \D(Q_k)^W
\subseteq \D(\vreg)^W\,$. Now, we recall the algebra embedding \eqref{HC},
which defines the Dunkl representation for the spherical subalgebra
of $ H_k $.
\begin{prop}
\label{is} The image of $\,\Res: \,U_k \into \D(\vreg)^W$  coincides with
$ \D(Q_k)^W $. Thus the Dunkl representation of $ U_k $ yields an algebra
isomorphism $\,U_k \cong \D(Q_k)^W$.
\end{prop}
\begin{proof}
In the Coxeter case, this is the result of \cite{BEG}, Proposition
7.22. In general, the proof is similar, {\it provided} the
results of the previous section are available. Indeed,
by Corollary~\ref{sp}, the image of $\, \Res \,$ is contained
in $ \D(Q_k)^W $. So we need only to see that the map
$\,\Res: \,U_k \to \D(Q_k)^W $ is surjective.
Passing to the associated graded algebras, we first note that
$\,\grd\,\D(Q_k)^W \subseteq \grd\,\D(V)^W $ by Lemma
\ref{symbol}~($\mathsf{ii}$). On the other hand, by the PBW
property \eqref{pbw} of $ H_k $, the Dunkl
representation induces an isomorphism $\grd\,U_k \cong
\grd\,\D(V)^W $. Hence, the associated graded map $\grd\,U_k \to
\grd\,\D(Q_k)^W$ is surjective, and so is the map
$\, U_k \to \D(Q_k)^W $.
\end{proof}
\begin{cor}\la{grcof}
$\ \grd\,\D(V) \,$ is a finite module over
$\,\grd\,\D(Q_k) \,$. Consequently $\,\grd\,\D(Q_k)\,$
is a finitely generated (and hence, Noetherian) commutative $\c$-algebra.
\end{cor}
\begin{proof}
We have already seen that $\, \grd\,\D(Q_k)^W \subseteq \grd\,\D(Q_k)
\subseteq \grd\,\D(V) \,$.
On the other hand, by Proposition~\ref{is}, $\,\grd\,\D(Q_k)^W =
\grd\,U_k = \grd\,U_0 = \grd\,\D(V)^W = [\grd\,\D(V)]^W \,$. Since
$W$ is finite, $\,\grd\,\D(V) \,$ is a finite module over $ [\grd\,\D(V)]^W $,
and hence {\it a fortiori}
over $ \grd\,\D(Q_k) $. This proves the first claim of the corollary.
The second claim follows from the first by the Hilbert-Noether Lemma.
\end{proof}

\begin{remark}
Following \cite{Kn}, let us say that an algebra $\, A \subseteq \D(\k) \,$
is {\it graded cofinite} in $ \D(V) $ if $ \grd\, A \subseteq \grd\,\D(V) $
and $ \grd\,\D(V) $ is a finite module over $ \grd\,A $. Under the
assumption that $\, A \subseteq \D(V)\,$, such algebras are described
in \cite{Kn}. Corollary~\ref{grcof} shows that $\,\D(Q_k)\,$ is graded cofinite
in $ \D(V) $, although it is actually {\it not} a subalgebra of $ \D(V)$.
It might be interesting to see whether the geometric description of graded
cofinite algebras given in \cite{Kn} extends to our more general situation.

Another interesting problem is to study the structure of $\,\grd\,\D(Q_k)\,$
as a module over $\,\grd\,\D(Q_k)^W \,$.
This is a natural ``double'' of the $ \c[V]^W$-module $ Q_k $.
In contrast to Theorem~\ref{chev}, the module $\,\grd\,\D(Q_k)\,$ is
not free over $\,\grd\,\D(Q_k)^W $, since $\, \D(Q_k) \,$
is not free over  $\, \D(Q_k)^W $ (see Corollary~\ref{Dchev} below).
\end{remark}
\subsection{Simplicity and Morita equivalence}
\la{Simple}
We now prove Theorem~\ref{ma} from the Introduction, which
is a generalization of \cite{BEG}, Theorem~9.7. Our proof
is similar to that of \cite{BEG}, except for the fact
that $ Q_k $ may not be a ring in general. We give some details
for completeness.

\begin{proof}[Proof of Theorem \ref{ma}]
First, by Theorem~\ref{morita}, $\, U_k \,$ is a
simple ring, and hence so is $\, \D(Q_k)^W $, by
Proposition~\ref{is}. An easy argument (see \cite{BEG}, p.~319)
shows that $\, \c[V]^W \cap\, I \not= \{0\} $ for any nonzero
two-sided ideal $ I $ of $ \D(Q_k) \,$. Since $\, \c[V]^W =
Q_k^W \subset \D(Q_k)^W $, we have $ \D(Q_k)^W \cap I \not=
\{0\}\,$ and therefore $\, 1 \in I \,$ (by simplicity of $ \D(Q_k)^W$).
This proves the simplicity of $ \D(Q_k)\,$.

Now, letting $\,\P := \{\,D \in \D(\k)\ |\ D(f) \in Q_k \
\text{for all} \ f \in \c[V]\,\}\,$, we note that $ \P \subset
\D(V) $ is a right ideal of $ \D(V) $, with $\, \End_{\D(V)}\,\P
\cong \D(Q_k)\,$. To see the latter, we can argue as in \cite{SS},
Proposition~3.3. First, it is clear that $\,\P\,$ is closed under
the left multiplication by the elements of $ \D(Q_k) $ in $ \D(\k)
$: this gives an embedding $\,\D(Q_k) \subseteq \End_{\D(V)}\,\P
\,$. On the other hand, $\,\P(\c[V]) = Q_k \,$, since the $
\D(Q_k) $-module $\, Q_k/\P(\c[V])\,$ has a nonzero annihilator
(containing $ \P $), and hence, must be $ 0 $, by simplicity of $
\D(Q_k)$. Identifying now $\, \End_{\D(V)}\,\P \cong \{\,D \in
\D(\k)\ |\ D.\P \subseteq \P\,\}\,$, we have $\, D(Q_k) =
D\,\P(\c[V]) \subseteq \P(\c[V]) = Q_k\,$ for any $\, D \in
\End_{\D(V)}\,\P \,$, whence $\, \End_{\D(V)}\,\P \subseteq
\D(Q_k)\,$.

Finally, since $ \D(V) $ and $ \D(Q_k) $ are both simple rings,
$\, \End_{\D(V)}\,\P \cong \D(Q_k)\,$ implies that $ \P $ is a
progenerator in the category of right $ \D(V)$-modules, and $
\D(V) $ and $ \D(Q_k) $ are Morita equivalent rings.
\end{proof}

As an interesting consequence of Proposition~\ref{is} and Theorem~\ref{ma}, we get
\begin{cor}\la{Dchev}
$ \D(Q_k) $ is a non-free projective module over $ \D(Q_k)^W $.
\end{cor}
\begin{proof}
Regarding $ \D(Q_k) $ as a right module over $\,\D(Q_k)^W $, we can identify
\begin{equation}
\la{end1}
\End_{\D(Q_k)^W}\,[\D(Q_k)] \cong  \D(Q_k) \ast W\ .
\end{equation}
Since $ \D(Q_k)^W $ and $ \D(Q_k) $ are simple rings, $\, \D(Q_k) \ast W $ is 
a simple ring, Morita equivalent to $\,\D(Q_k)^W $ (see \cite{M}, Theorem~2.4). 
It follows from \eqref{end1} that $ \D(Q_k) $ is a progenerator in the category 
of right $ \D(Q_k)^W$-modules; in particular, $ \D(Q_k) $ is a f.~g. projective 
module over $ \D(Q_k)^W $.

To prove that $ \D(Q_k) $ is not free over $ \D(Q_k)^W $, we will
use a $K$-theoretic argument identifying the Grothendieck group
of $ \D(Q_k) \ast W $ in two different ways. First, we have
\begin{equation}
\la{end2}
K_0\,(\D(Q_k) \ast W) \cong K_0\left(\D(Q_k)^W\right) \cong
K_0(U_k) \cong K_0(H_k) \cong K_0(\c W)\ ,
\end{equation}
where the first isomorphism is induced by the above Morita equivalence 
between $ \D(Q_k)\ast W $ and $\, \D(Q_k)^W$, the second by the 
isomorphism $\,\D(Q_k)^W \cong U_k \,$ of Proposition~\ref{is}, the third by 
the Morita equivalence between $ U_k $ and $ H_k $ (see Theorem~\ref{morita}) 
and the last by the natural embedding $\,\c W \into H_k\,$.

Second, by Theorem~\ref{ma}, the ring $ \D(Q_k) $ is Morita
equivalent to $ \D(V) $ via the progenerator $ \P = \D(V,\,Q_k)$. 
Extending $ \tilde{\P} := \P \otimes_{\D(V)} (\D(V) \ast W) $, it
is easy to see that $ \tilde{\P} $ is a progenerator from
$ \D(V) \ast W $ to $ \D(Q_k) \ast W $. Hence, we get isomorphisms
\begin{equation}
\la{end3}
K_0\,(\D(Q_k) \ast W) \cong K_0\,(\D(V) \ast W) \cong K_0(\c W)\ ,
\end{equation}
where the last one is induced by the inclusion $\,\c W \into \D(V) \ast W\,$.

Now, suppose that $ \D(Q_k) $ is a free module over $ \D(Q_k)^W $
of rank $ r > 1 $ (say). Then, under the isomorphism \eqref{end2}, 
the class of the free module $\,\D(Q_k) \ast W\,$  
corresponds to $r$ times the class of the trivial representation of $ W $; 
in particular, the image of $\,[\D(Q_k) \ast W]\,$
under \eqref{end2} is divisible by $ r $ in $\, K_0\,(\c W)\,$.
On the other hand, under \eqref{end3}, the class $\,[\D(Q_k) \ast W]\,$
corresponds to the class of the regular representation 
$\, [\c W] \in K_0(\c W) $. Since $\, K_0(\c W) = \bigoplus_{\tau \in \W} \Z \cdot [\tau] \,$ 
is a free abelian group based on the classes of irreducible representations, $\, [\c W] $ 
is obviously not divisible in $ K_0(\c W) $ by any integer greater than one. 
Thus, we arrive at contradiction which proves that $ \D(Q_k) $ cannot be a free module 
over $ \D(Q_k)^W $.
\end{proof}

\section{Shift Operators}
\la{ShOper}
\subsection{Automorphisms of $\mathcal DW$} We start by describing
certain automorphisms of the algebra $\mathcal \D W$ and their action on
the subalgebras $H_k$ and $U_k=\e H_k\e$. Recall that $ \D W$ is generated
by the elements $\,w \in W \,$, $\, x \in V^*\,$ and
$ \xi \in V $, so any automorphism of $ \D W $ is determined
by its action on these elements.

Given a one-dimensional character $\chi$ of $W$, we define our first automorphism by
\begin{equation}\la{fchi}
w\mapsto\chi(w)w\,,\quad x\mapsto x\,,\quad
\partial_\xi\mapsto\partial_\xi
\,.
\end{equation}
Under \eqref{fchi}, the subalgebras $H_k$ and $U_k$ transform to
$H_{k'}$ and $\e_\chi H_{k'}\e_\chi$, where
$\,\e_\chi\in\c W$ is the idempotent corresponding to $\chi$, and
$\,k'_{H,i} := k_{H,i+a_H} $ with $\, a_H \in \Z \,$ determined by
$\,\chi|_{W_H}=(\det)^{a_H}\,$.

To define the second automorphism we fix a $W$-orbit
$C\subseteq \A $ and a $W$-invariant
closed $1$-form $\omega$ on $\vreg$:
\begin{equation}\la{om}
\omega=\lambda\,d\,\log \delta_C=\lambda\sum_{H\in
C}\frac{d\,\alpha_H}{\alpha_H}\ ,\quad \lambda \in \c\ .
\end{equation}
Then, regarding $\,\xi\in V\,$ as a constant vector field on $\vreg$,
we define
\begin{equation}\la{gomega}
w\mapsto w\,,\quad x\mapsto x\,,\quad
\partial_\xi\mapsto\partial_\xi+\omega(\xi)
\,,
\end{equation}
This automorphism maps the algebras $H_k$ and $U_k$ to $H_{k'}$ and $U_{k'}$,
where $k'$ is given by $ k'_{C,i}=k_{C,i}-\lambda/n_{C}$ and $k'_{C',i}=k_{C',i}$ for $C'\ne C$.

Finally, for a fixed $\,C \in \A/W \,$, we consider the
automorphism $\,u \mapsto \delta_C u\,\delta_C^{-1} \,$ given by
conjugation by the element \eqref{delc}. It is easy to
see that this automorphism is the composition of the automorphism
\eqref{fchi}, with $\chi=\det_C^{}$, and the automorphism
\eqref{gomega}, with $\lambda=-1$. Therefore, it maps
$H_k$, $U_k$ to $H_{k'}$ and
$\epsilon_C\,H_{k'}\epsilon_C$, where
\begin{equation}\la{edelta}
\epsilon_C=\delta_C^{}\,\e\,\delta_C^{-1}=|W|^{-1}\sum_{w\in W}
(\mathrm{det} _C^{}w)w\,,
\end{equation}
and $k'$ is related to $k$ by
\begin{equation}\la{cdelta}
k'_{C,i}=k_{C,i+1}+1/n_C\quad\text{and}\quad k'_{C',i}=k_{C',i}\
\,\text{for}\ C'\ne C\,.
\end{equation}

\subsection{Twisted quasi-invariants} For the purposes of this section,
we redefine quasi-invariants in
a slightly greater generality to allow fractional multiplicities.
Precisely, we fix a $W$-invariant function $\,a: \A \to \Z$ and
choose $\,k_{C,i} \in \Q \,$ so that
\begin{equation}\la{co}
k_{C,i}\equiv a_C/n_C\ \mod\, \Z\ .
\end{equation}
(In particular, $\,a=0\,$ corresponds to the case of integral
$k$'s.) For such $k$, we take $\,Q_k\,$ to be the subspace of all
$\, f \in \c[\vreg] \,$ satisfying
\begin{equation}
\label{qc3} \e_{H,-i-a_H} (f) \equiv 0\
\mbox{mod}\,\langle\alpha_H^{n_Hk_{H,i}}\rangle
\end{equation}
for all $H\in \A $ and $i=0,1,\ldots, n_H-1$. In the
case of negative multiplicities,
$\langle\alpha_H^{n_Hk_{H,i}}\rangle$ should be understood as the
span of rational functions $\,f\in \c[\vreg]\,$ for which
$f\cdot\alpha_H^{-n_Hk_{H,i}}$ is regular along $H$ (although it
may still have poles along other hyperplanes).

The proof of Theorem \ref{Qfat} will work in this more general situation,
if we modify the definition of $\QQ_k\subset \c[\vreg]\otimes \c W $ in the following way, cf. \eqref{qcc}:
\begin{equation}
\label{qcn3}
\varphi \in \QQ_k \quad \Longleftrightarrow\quad (1\otimes\e_{H,i+a_H}) \varphi \equiv 0\
\mbox{mod}\,\langle\alpha_H^{n_Hk_{H,i}}\rangle\otimes\c W
\end{equation}
for all $H\in \A $ and $i=0,1,\ldots, n_H-1$.
\begin{example}\la{ex}
Let $\,W=\Z/n\Z \,$ and suppose that $\, k_i\equiv a/n\ (\mod\,
\Z)$. In that case, we have
\begin{equation}\la{nc1}
\QQ_k=\bigoplus_{i=0}^{n-1} x^{nk_i}\c[x]\e_{i+a}\
,\quad\,\e_i=\frac{1}{n}\sum_{w\in W}(\det w)^{-i} w\,.
\end{equation}
On the other hand, it is easy to see that the subspace $\,Q_k
\subseteq \c[V] \,$ is still described by formula \eqref{exq}, which
is actually independent of $a$. As a consequence, for different
values of $ k $, we may get the same $Q_k$. For example, if we take
$ k'$ to be
\begin{equation}
\label{g} k'_{i}=k_{i-1}-\frac{1}{n}\ \text{for}\ i=1,\dots,n-1\,,\
k'_{0}= k_{n-1} - \frac{1}{n}+1\ ,
\end{equation}
then the formula \eqref{exq} gives that $\,Q_{k'}=Q_k\,$. More
generally, this holds for all iterations of \eqref{g}, which form a
cyclic group of order $n$. In the next section, we extend this
observation to an arbitrary group $W$.
\end{example}

\bigskip

\subsection{Symmetries of the Dunkl representation}
The Dunkl representation defines a flat family of subalgebras $ \{U_k\} $
of $ \D(\vreg)^W $, with $\grd(U_k)=\c[V\times V^*]^W$ for any $k$.
It turns out that
this family is invariant under a certain subgroup $G$ of affine
transformations of $k$, so that $\,U_k = U_{k'}\,$ whenever $k'=g\cdot k$
with $g\in G$. This kind of invariance is not obvious from definitions:
we will deduce it by studying the action of $ G $
on modules $Q_k$ of quasi-invariants.

\medskip

First, as in Example~\ref{ex}, for $\,C\in \A/W\,$ we define
the transformation $\,g_C: k\mapsto k'\,$ by
\begin{equation}
\label{gh} k'_{C, i}=k_{C, i-1}-\frac{1}{n_C}+\delta_{i,0}\
\text{and}\quad k'_{C',i}=k_{C',i}\quad\text{for}\ C'\ne C\,.
\end{equation}
Note that $(g_C)^{n_C}= \id $. Note also that if $k$ satisfies
the conditions \eqref{co}, then $k'$ satisfies the same
conditions, with $a$ replaced by $ a' := a-1_C $, where $1_C\,: \A\to \Z$
is the characteristic function of the orbit $C$.

\begin{prop}\la{G} Let $G$ denote the (abelian) group generated by all
$g_C$ with $C\in \A/W$. Then $Q_{k'}=Q_k$ for any $k'\in G\cdot
k$, provided $k$ satisfies \eqref{co}.
\end{prop}
\begin{proof}
A straightforward calculation shows that the two systems of congruences
\eqref{qc3} for $k$ and $\,k' = g_C \cdot k\,$ are equivalent.
As in Example \ref{ex} above, this implies the
equality $ Q_{k'} = Q_k $.
\end{proof}
\medskip

For the purposes of Section \ref{MQ}, we will need an analogue of
the above result for the modules of $\tau$-valued quasi-invariants
$\QQ_k(\tau)$. First, we need to modify their definition similarly
to \eqref{qcn3}:
\begin{equation}
\label{qcn31} \varphi \in \QQ_k(\tau) \quad \Longleftrightarrow\quad
(1\otimes\e_{H,i+a_H}) \varphi \equiv 0\
\mbox{mod}\,\langle\alpha_H^{n_Hk_{H,i}}\rangle\otimes\tau
\end{equation}
for all $H\in \A $ and $i=0,1,\ldots, n_H-1$. Then it is easy to
see that $\QQ_k(\tau)$ can be described similarly to \eqref{qtauh}:
\begin{equation}\label{qttauh}
\QQ_k(\tau)=\bigcap_{H \in \A}\QQ_k^H(\tau)\,,\quad
\QQ_k^H(\tau)=\bigoplus_{i=0}^{n_H-1}
\langle\alpha_H\rangle^{n_Hk_{H,i}}\,\otimes \e_{H,i+a_H}\tau\,.
\end{equation}
As before, the space $\QQ_k(\tau)\subset \c[\vreg]\otimes\tau$ is
invariant under the differential action of $H_k$. As a result, the
subspace $\e\QQ_k(\tau)$ of $W$-invariant elements in $\QQ_k(\tau)$
becomes a module over the spherical subalgebra $\e H_k \e$.
Furthermore, the proof of Lemma \ref{inter} applies verbatim, so we
have the formula
\begin{equation}\label{eho11}
\e\QQ_k(\tau)=\bigcap_{H\in\A}\e_{H,0}\QQ_k^H(\tau)\,,
\end{equation}
with each of the subspaces $\e_{H,0}\QQ_k^H(\tau)$ described
similarly to \eqref{eho1}:
\begin{equation}\label{eho12}
\e_{H,0}\QQ_k^H(\tau)=\bigoplus_{i=0}^{n_H-1}
\alpha_H^{n_Hk_{H,i}+i}\,\c[\vreg^H]^{W_H}\otimes\e_{H,i+a_H}\tau\,.
\end{equation}
Finally, using \eqref{eho11} and \eqref{eho12}, we obtain similarly
to Proposition \ref{G} the following result.

\begin{prop}\la{Gtau} Let $G$ denote the abelian group generated by all
transformations \eqref{gh}. Then for any $k$ satisfying \eqref{co}
and any $k'\in G\cdot k$, we have $\e\QQ_{k'}(\tau)=\e\QQ_k(\tau)$
as subspaces in $\c[\vreg]\otimes\tau$.
\end{prop}


\medskip

Proposition~\ref{G} has the following important consequence.

\begin{prop}\label{shift1} Let $k$ be arbitrary and $k'\in G\cdot k$.
Then the  sphe\-ri\-cal subalgebras $U_k=\e H_k\e$ and
$U_{k'}=\e H_{k'}\e$ coincide as subsets in $\D W$ and hence are
isomorphic. Furthermore, we have $\,\e T_{p,k}\e=\e T_{p,k'}\e\,$ for
any $p\in\c[V^*]^W$, or equivalently, $\,L_{p, k}=L_{p, k'}\,$, where
$L_{p,k}:=\Res\,(\e T_{p,k}\e)$.
\end{prop}

\begin{proof} First, we prove the claim under the integrality
assumption \eqref{co}. By Proposition \ref{G}, we have $Q_k=Q_{k'}$,
so that $\D(Q_k)^W=\D(Q_{k'})^W$. On the other hand, Proposition \ref{is}
says that $U_k=\e\D(Q_k)^W$ and $U_{k'}=\e\D(Q_{k'})^W$. Whence $U_k=U_{k'}$.

To prove the second claim, let $L,\,L'$ denote $L_{p, k}$ and $L_{p,
k'}$, respectively. From the defintion of the Dunkl operators it
easily follows that $L$ and $L'$ have the same principal symbol
$p(\partial)$, and their lower order coefficients are rational
functions of negative homogeneous degrees. Hence $ L - L' $ is a
differential operator whose all coefficients have negative
homogeneous degree. But, by Proposition \ref{G} and Theorem
\ref{sp}, both $L$ and $ L'$ are in $\mathcal D(Q_k)$, so, by Lemma
\ref{symbol}($\mathsf{ii}$), the principal symbol of $\, L - L' \,$
must be regular. This proves that $\,L=L'$.

To extend the above results to arbitrary $k$, take $k'=g\cdot k$,
with {\it fixed} $g\in G$. For the standard filtration, we have
$\grd\, U_k =\grd\, U_{k'} \cong \c[V\times V^*]^W$. Thus, we may view
$\,k\mapsto U_k\,$ and $\,k\mapsto U_{g\cdot k}\,$ as two flat
families of filtered subspaces in $\D W$. We know that these
subspaces coincide when $k$ takes rational values satisfying
\eqref{co}. Since the set of such values of $k$ is Zariski dense in
the space of all complex multiplicities, we conclude that
$U_k=U_{g\cdot k}$ holds for all $k$. In the same spirit, we have
$\,L_{p,k}=L_{p,k'}\,$ for rational $k$, and both sides of this
equality depend polynomially in $k$, hence the same must be true for
all $k$.
\end{proof}

\subsection{Isomorphisms of spherical algebras}
In this section, we will regard $\,k=\{k_{C,i}\}$ as a vector in $\c^N$, with
$N={\sum_{C\in \A/W}n_C}$. Let $\{\ell_{C,i}\}$ denote the
standard basis in this vector space, so that
$k=\sum_{C\in\A/W}\sum_{i=0}^{n_C-1} k_{C,i}\ell_{C,i}$.
(As usual, we assume $\ell_{C,i}$ to be periodic in $i$, so that
$\ell_{C,n_C}=\ell_{C,0}$.)

The next proposition describes the transformation of $U_k$ under translations
$\,k \mapsto k+ \ell_{C,n_C-1}$. In the Coxeter case, this result was first
established in \cite{BEG} for generic ('regular') multiplicities and later extended
in \cite{G} to arbitrary $k$'s when $W$ is crystallographic. We now prove it in full
generality: for an arbitrary complex reflection group and arbitrary multiplicities.

\begin{prop}\la{shiso}
For a fixed $\, C \in \A/W \,$, we have the following isomorphisms

$(1)$ $\e H_k\e \cong\epsilon_C H_{k'}\epsilon_C$,
$\ k'=k+\ell_{C,n_C-1}$;

$(2)$ $\e H_k\e\cong \epsilon H_{k'}\epsilon$,
$\ k'=k+\sum_{C\in\ms{A}/W}\ell_{C,n_C-1}$\ ,

\noindent
where $\epsilon $ is the sign
idempotent on $W$ and $ \epsilon_C $ is given by \eqref{edelta}.
\end{prop}

\begin{proof}

Let $f=f_C$ and $g=g_C$ be the transformations $k\mapsto k'$ defined
by \eqref{cdelta} and \eqref{gh}, respectively. Recall that $f$
describes the effect of the conjugation by $\delta_C$, so that
\begin{equation}\la{ff}
\delta_C\,T_{\xi,k}\,\delta_C^{-1}=T_{\xi, f(k)}\quad\text{and}\quad
\delta_C\,H_k\,\delta_C^{-1}=H_{f(k)}\,.
\end{equation}
On the other hand, by Proposition \ref{shift1}\,, $\e H_k\e=\e
H_{g(k)}\e$. Now, a simple calculation shows that
$k':=fg(k)=k+\ell_{C,n_C-1}$. Combining all these together, we
get
$$
\e H_{k}\e=\e
H_{g(k)}\e=\e\,\delta_C^{-1}H_{fg(k)}\delta_C\,\e
=\delta_C^{-1}\,\epsilon_C H_{k'}\epsilon_C\,\delta_C \cong
\epsilon_C H_{k'}\epsilon_C\ ,
$$
which is our first isomorphism. The second isomorphism is
proved in a similar way, using $f=\prod_{C\in\ms{C}}f_C$
and $g=\prod_{C\in\ms{C}}g_C$ instead of $f_C,\,g_C$.
\end{proof}

Note that the above proof gives a bit more than stated in the
proposition: it shows that $\e H_k\e = \e \delta_C^{-1}
H_{k'}\delta_C \e$ as {\it subsets} in $\D W$. Now, arguing as in
(the proof of) Proposition \ref{shift1}, we conclude that
$\e\,T_{p,k}\,\e\,=\,\e\,\delta_C^{-1}T_{p,k'}\delta_C\,\e$ for any
$W$-invariant polynomial $p$. More generally, we have the following
result, which answers a question of Dunkl and Opdam (see \cite{DO},
Question 3.22).

\begin{prop}\la{shisoa}
For fixed $\,C\in\ms{A}/W\,$ and $\,a=1,\ldots, n_C-1\,$, let
\begin{equation}\la{mrel}
k'=k+\sum_{i=1}^a \ell_{C,n_C-i}\ .
\end{equation}
Then $\,\e H_k\e = \e \delta_C^{-a} H_{k'}\delta_C^a \e\,$ in $\, \D W $, and
$\,\e\,T_{p,k}\,\e\,=\,\e\,\delta_C^{-a}T_{p,k'}\delta_C^a\,\e\,$ for all
$\,p\in\c[V^*]^W$.
\end{prop}

This is proved by replacing the transformations $f=f_C$, $g=g_C$ in
the proof of Proposition~\ref{shiso} by their iterates, $f^a$ and
$g^a$.{\hfill $\square$}

\bigskip

\subsection{Shift operators}
\la{shop}
We are now in position to construct the Heckman-Opdam shift operators for
the group $W$, extending an idea of G.~Heckman \cite{H}. Fix $C\in\ms{A}/W$ and
$a\in\{0,\dots, n_C-1\}$ as above, and recall the elements $\delta_C$,\
$\delta_C^*$, see \eqref{delc}. For an arbitrary $k$, define $k'$ by
\eqref{mrel} (with $k':=k$ in the case $a=0$) and introduce the following
differential operators
\begin{equation}
\la{s}
S_k :=\Res \left(\delta_C^{1-a}T_{\delta_C^*,k'}\delta_C^{a}\right)\,,\quad
S_k^- := \Res \left(\delta_C^{-a}(T_{\delta_C^*,k'})^{n_C-1}\delta_C^{a-1}\right)\,.
\end{equation}
Note that both expressions under $\,\Res\,$ are $W$-invariant.

\begin{theorem}\la{shiftoper}
For all $\,p\in\c[V^*]^W$, the operators $S_k$ and $ S_k^-$ satisfy the following intertwining relations
\begin{equation*}
L_{p,\widetilde k}\circ S_k=S_k\circ L_{p,k}\ , \quad L_{p,k}\circ
S_k^-=S_k^-\circ L_{p,\widetilde k}\ ,
\end{equation*}
where $\,\widetilde k=k+\ell_{C,n_C-a}\,$.
\end{theorem}
\begin{proof} Let $f=f_C$ and $g=g_C$ be the same as in the proof of Proposition \ref{shiso},
and let $k'$ be as in \eqref{mrel}. A direct calculation shows that $k'=f^{a}g^{a}(k)$ and
$g^{1-a}fg^{a}(k)=k+\ell_{C,n_C-a}$. As a result, if we let $\, k_1:=f^{1-a}(k') \,$
and $\, k_2:=f^{-a}(k')\,$, then $\,\widetilde{k} = g^{1-a}(k_1)\,$ and $\,k_2=g^{a}(k)\,$.
By Proposition \ref{shift1}, this implies
\begin{equation*}
L_{p,\widetilde k}=L_{p,k_1}\quad\text{and}\quad
L_{p,k}=L_{p,k_2}\,.
\end{equation*}
To prove the first identity it thus suffices to show that $S_k$
intertwines $L_{p,k_1}$ and $L_{p,k_2}$. Writing $\delta \,$, $\delta^*$ for $\delta_C $, $\delta_C^*$,
we have
\begin{equation*}
\begin{array}{llll}
\e L_{p,k_1}S_k &=& \e T_{p,k_1}\delta^{1-a}T_{\delta^*, k'}\delta^{a}\e &\quad   \\*[.8ex]
                &=& \e\delta^{1-a}T_{p,f^{a-1}(k_1)}T_{\delta^*, k'}\delta^{a}\e &\quad \mbox{(by \eqref{ff})}  \\*[.8ex]
                &=& \e\delta^{1-a}T_{p,k'}T_{\delta^*,k'}\delta^{a}\e & \quad \\*[.8ex]
                &=& \e\delta^{1-a}T_{\delta^*,k'}T_{p,k'}\delta^{a}\e & \quad \mbox{(by Lemma~\ref{duprop}$\mathsf{(i)}$)}  \\*[.8ex]
                &=& \e\delta^{1-a}T_{\delta^*,k'}\delta^{a}T_{p,f^{-a}(k')}\e & \quad   \\*[.8ex]
                &=& \e\delta^{1-a}T_{\delta^*,k'}\delta^{a}\e\cdot \e T_{p,k_2}\e & \quad \mbox{(by Lemma~\ref{duprop}$\mathsf{(ii)}$)} \\*[.8ex]
                &=& \e S_k L_{p, k_2}\ . &
\end{array}
\end{equation*}
The second identity involving $S_k^-$ is proved in a similar fashion.
\end{proof}

\section{Category $\mathcal{O}$}
\la{o}
Throughout this section, we will use the following notation: if
$ A $ is an algebra, we write $ \Mod(A) $ for the category
of all left modules over $A$, and $ \mod(A) $ for its
subcategory consisting of finitely generated modules. In
particular, when $ A $ is a finite-dimensional algebra over
$\c$ (e.g., $ A= \c W $), $\,\mod(A)\,$ is the category of
finite-dimensional modules over $A$.

\subsection{Standard modules}
Recall that the Cherednik algebra $ H_k = H_k(W) $ admits a
decomposition $\, H_k \cong \c[V] \otimes \c W \otimes \c[V^*] \,$,
which is similar to the PBW decomposition $\,U(\g) \cong U(\n_{-})
\otimes U(\h) \otimes U(\n_+)\,$ for the universal enveloping
algebra of a complex semisimple Lie algebra $ \g $.  This suggests
to view the subalgebras $ \c[V] $, $ \c[V^*] $ and $ \c W $ of $ H_k $
as analogues of $ U(\n_{-})$, $ U(\n_+) $ and $ U(\h) $ respectively,
and introduce a category of `highest weight modules' over $ H_k $
by analogy with the Bernstein-Gelfand-Gelfand category $ \O_{\g} $
in Lie theory.

Precisely, the category $ \O_k := \O_{H_k} $ is defined as
the full subcategory of $ \mod(H_k) $, consisting of modules
on which the elements of $\, V \subset \c[V^*] $ act locally nilpotently:
\begin{equation*}
 \O_k :=\,\{M \in \mod(H_k)\ :\ \xi^d m = 0\ ,\
 \forall\, m\in M\,,\ \forall\,\xi\in V\, ,\ \forall\,d \gg 0\}\ .
\end{equation*}
It is easy to see that $\O_k$ is closed under taking subobjects, quotients and
extensions in $ \mod(H_k) $: in other words, $ \O_k $ is a Serre subcategory
of $ \mod(H_k) $.

The structure of $ \O_k $ is determined by so-called standard
modules, which play a r\^ole similar to Verma modules in Lie theory.
To define such modules we fix an irreducible representation $ \tau $ of $W$
and extend the $W$-structure on $ \tau $ to a $\,\c[V^*] * W$-module
structure by letting $\,\xi \in V \,$ act trivially.
The {\it standard $H_k$-module of type} $ \tau $ is then given by
\begin{equation}\la{sm}
M(\tau) := \mathrm{Ind}^{H_k}_{\c[V^*]*W}\,\tau\, =\,
H_k\underset{\c[V^*]*W}{\otimes} \tau\,.
\end{equation}
It is easy to see from the relations of $ H_k $ that
$\,M(\tau)\in \O_k\,$. Moreover, the PBW theorem
\eqref{pbw} implies that $\, M(\tau) \cong \c[V]\otimes \tau \,$
as a $\c[V]$-module.

The basic properties of standard modules are summarized in the following
\begin{prop}
\la{ind}
Let $\, \W $ be the set of irreducible representations of $ W $.

$(1)$ $\{M(\tau)\}_{\tau \in \W} $ are pairwise non-isomorphic
indecomposable objects of $ \O_k $.

$(2)$\ Each $M(\tau)$ has a unique simple quotient $L(\tau)$, and
$\{L(\tau)\}_{\tau\in\W}$ is a complete set of simple objects of
$ \O_k $.

$(3)$\ Every module $M \in \O_k $ admits a finite filtration
\begin{equation}
 \la{filtm}
 \{0\} = F_0\subset F_1\subset \ldots \subset F_N = M  \ ,
\end{equation}
with $F_i\in \O_k $ and $\,F_i/F_{i-1}\cong L(\tau_i)\,$ for some
$\tau_i\in\W$.
\end{prop}
\begin{proof} The first claim follows from \cite{DO}, Proposition 2.27 and
Corollary 2.28. The second and the third are \cite{GGOR}, Proposition 2.11 and Corollary 2.16, respectively.
\end{proof}

\subsection{The Knizhnik-Zamolodchikov (KZ) functor}\la{lkz}
Introduced by Opdam and Rouquier, this functor is one of the main tools
for studying the category $\O$. We briefly review its construction
referring the reader to \cite{GGOR} for details and proofs.

First, using Proposition~\ref{Loc}, we introduce the localization functor
\begin{equation}
\la{locf}
\Mod(H_k) \to \Mod(\D W)\ , \quad M \mapsto \Mreg := \D W \underset{H_k}{\otimes} M\ .
\end{equation}
By definition, $\, \Mod(\D W) $ is the category of $W$-equivariant
$\D$-modules on $ \vreg $. Since $W$ acts freely on $\vreg$, this category is
equivalent to the category $\,\Mod\,\D(\vreg/W) \,$ of $\D$-modules on the
quotient variety $\vreg/W$.  The full subcategory of  $\,\Mod\,\D(\vreg/W) \,$
consisting of $ \O$-coherent $ \D$-modules  is equivalent to the category
of vector bundles on $ \vreg/W $ equipped with a regular flat connection, which
is, in turn, equivalent to the category of finite-dimensional representations
of the Artin braid group $\, B_W :=\pi_1(\vreg /W,\,*\,)$ (the Riemann-Hilbert correspondence).

Now, in view of Proposition~\ref{ind}, localizing an object in the category
$\, \O_k \subset \Mod(H_k)\,$ yields a $\D W$-module, which is finite over
$ \c[\vreg]$. Hence, combined with above equivalences, the restriction of
\eqref{locf} to $\, \O_k \,$ gives an exact additive functor
\begin{equation} \la{KZZ}
\KZ_k \,:\ \O_k\  \to \ \mod(\c B_W)\ .
\end{equation}

We illustrate this construction by applying \eqref{KZZ} to a
standard module $M=M(\tau)$ (cf. \cite{BEG}, Prop. 2.9). Since $
M(\tau) \cong \c[V] \otimes \tau $ as a $\c[V]$-module, $\,\Mreg $
can be identified with $ \c[\vreg]\otimes\tau\,$ as a
$\c[\vreg]$-module and thus can be thought of as (the space of
sections of) a trivial vector bundle on $\vreg$ of rank $\dim\tau$.
With this identification, the $\D$-module structure on $\Mreg$ is
described by
\begin{equation} \la{fov}
    \partial_\xi(f\otimes v)=\partial_\xi(f)\otimes v + f\otimes
    \partial_\xi(v)\ , \quad\forall\,\xi\in V\ ,
\end{equation}
where $\,f \in \c[\vreg] \,$ and $ v \in \tau $.
Since $\xi\, v = 0$ in $M$ and $\xi$ corresponds under localization
to the Dunkl operator $T_\xi$, we have $ T_\xi(v)= 0 $, or equivalently
\begin{equation}\la{dune}
\partial_\xi v -\sum_{H\in \A}\frac{\alpha_H(\xi)}{\alpha_H}
\sum_{i=0}^{n_H-1}n_Hk_{H,i}\e_{H,i}(v)=0\ ,\quad \forall\,\xi \in V\ .
\end{equation}
The relations \eqref{fov} can thus be rewritten as
\begin{equation}\la{conn}
\partial_\xi(f\otimes v) = \partial_\xi(f)\otimes v + \sum_{H\in \A}
\frac{\alpha_H(\xi)}{\alpha_H}\sum_{i=0}^{n_H-1}n_Hk_{H,i}f\otimes
\e_{H,i}v\ ,
\end{equation}
which gives an explicit formula for a regular flat connection
on $ \Mreg = \c[\vreg] \otimes \tau $. This connection is
called a {\it KZ connection} with values in $ \tau \,$: its
horizontal sections $\,y: \vreg \to \tau \,$ satisfy the
following KZ equations
\begin{equation}\label{kz}
\partial_\xi y + \sum_{H\in\mathcal
A}\frac{\alpha_H(\xi)}{\alpha_H}\sum_{i=0}^{n_H-1}n_Hk_{H,i}e_{H,i}(y)=0
\ ,\quad \forall\,\xi\in V\ .
\end{equation}

\begin{remark}
Notice a formal similarity between the systems \eqref{dune} and
\eqref{kz}. Apart from inessential change of sign, there is, however, an important
difference: in \eqref{kz}, the group elements $w\in W$  act on the {\it values} of
the functions involved, while in \eqref{dune} on their arguments.
\end{remark}

It is easy to check that if $y$ is a local solution of \eqref{kz} near a point
$x_0\in\vreg$, then $^wy:=wyw^{-1}$ is a local solution near $wx_0$. Thus, the
system \eqref{kz} is $W$-equivariant and descends to a regular holonomic
system on $\vreg/W$. The space of local solutions of this holonomic system
has dimension $\dim\tau$, and its monodromy gives a linear representation of the
braid group $\,B_W \,$ in this space. The corresponding $ \dim(\tau)$-dimensional
$\c B_W$-module is the value of the functor \eqref{KZZ} on $ M(\tau)$.
We remark that for
complex reflection groups, the system \eqref{kz} and its monodromy
have been studied in detail in \cite{K}, \cite{BMR} and \cite{O}.

\subsection{The Hecke algebra}\la{oh}
It is crucial for applications that the KZ functor \eqref{KZZ}
factors through representations of the Hecke algebra of $W$.
To define this algebra, we recall that, for
every $H\in \A$, there is a unique reflection $s_H\in W_H$ with
$\det s_H=\exp 2\pi i /n_H$. It is known that the
braid group $B_W$ is generated by the elements $\sigma_H$ which correspond
to $s_H$ as generators of monodromy around $H\in \A $  (see \cite{BMR}).
Given now complex parameters $k=\{k_{H,i}\}$, with
$k_{H,0}=0$, the Hecke algebra $ \H_k(W)$ is defined
as the quotient of $\c B_W$ by the following relations
\begin{equation*}
    \prod_{j=0}^{n_H-1}\left(\sigma_H-(\det s_H)^{-j}e^{2\pi i
    k_{H,j}}\right)=0\ ,\quad\forall\, H\in\ms A\ .
\end{equation*}
Notice that, for $k_{H,j} \in \Z $, these relations become $(\sigma_H)^{n_H}=1$, so
in that case $ \H_k(W)$ is canonically isomorphic to the
group algebra of $ W$. In general, $\H_k $ should be viewed as a deformation of
$\c W$.

Restricting scalars via the natural projection $\, \c B_W \onto \H_k(W) \,$, we can
regard $ \mod(\H_k) $ as a full subcategory of $ \mod(\c B_W) $. It turns out that
\begin{theorem}[\cite{GGOR}, Theorem 5.13]
For each $k$, the KZ functor \eqref{KZZ} has its image in $ \mod(\H_k) $, i.~e.
$\, \KZ_k :\, \ms O_k  \to  \mod(\H_k)\,$.
\end{theorem}
The next two results require the assumption that $\,\dim \H_k = |W|\,$.
It will be crucial for us that this assumption holds automatically for
{\it all} $\, W$ whenever $k_{H,j}\in\Z$, since $\,\H_k \cong \c W\,$
in this case\footnote{In general, the equlaity $\,\dim \H_k = |W|\,$ is known to
be true for almost all complex reflection groups, except for a few exceptional ones,
in which case it still remains a conjecture (see \cite{BMR}).}.

Let $\Otor$ denote the full subcategory of $\ms O_k $ consisting of modules $M$ such that
$\Mreg=0$. Clearly, $ \Otor $ is a Serre subcategory of $ \ms O_k $, so that
the quotient $\,\ms O_k/\Otor $ is defined as an abelian category.
\begin{prop}[\cite{GGOR}, Theorem 5.14]\la{tor} Assume that $\dim \H_k =|W|$. Then
the KZ functor induces an equivalence
\begin{equation*}
    \KZ_k\,:\ \ms O_k/\Otor   \overset{\sim}\to \mod(\H_k)\ .
\end{equation*}
\end{prop}
In addition, one can prove
\begin{theorem}[\cite{GGOR}, Theorems 5.15, 5.16]\la{dct}  Assume that $\dim
\H_k =|W|$. Then there exist projective objects $P\in\ms O_k$ and
$Q\in\mod(\H_k)$ such that
\begin{equation*}
    \H_k \cong \left(\End_{\ms
    O_k}P\right)^{\mathrm{opp}}\quad\text{and}\quad \ms O_k \simeq
    \mod\,\left(\End_{\H_k}Q\right)^{\mathrm{opp}}\ .
\end{equation*}
\end{theorem}

\subsection{Regularity}\la{SReg}

The structure of the category $\ms O_k$ depends on the values of the
parameters $k$. For generic $k$'s, $\,\ms O_k\,$ is a semisimple category,
while for special values of $k$ it has a more complicated structure
(in particular, it has homological dimension $ > 0 $). Likewise, the Hecke algebra $\H_k$ is
semisimple for generic $k$'s, but becomes more complicated for certain special values.
Using the KZ functor, we can show that the special values in both cases actually coincide.
Precisely, we have the following
\begin{theorem}\la{ss}
Assume that $\dim\H_k=|W|$. Then the following are equivalent.

1. $\H_k$ is a semisimple algebra.

2. $\ms O_k$ is a semisimple category.

3. $H_k$ is a simple ring.
\end{theorem}

\begin{proof} We give a detailed proof of this result following R.Vale's dissertation \cite{Va}
(cf. {\it loc. cit.}, Theorem~2.1).

${\it 1 \Rightarrow 2}$. Choose $ Q\in \mod(\H_k) $ as in Theorem~\ref{dct}. If
$\H_k $ is a semisimple algebra, then $\,\End_{\H_k}Q\,$ is a semisimple algebra, and hence $\, \O_k \simeq
\mod\,\left(\End_{\H_k}Q\right)^{\mathrm{opp}} $ is a semisimple category.

${\it 2 \Rightarrow 1}$. Choose $\, P\in\ms O_k \,$ as in Theorem~\ref{dct}.
If $ \O_k $ is a semisimple category, then $\,\End_{\ms O_k}P $ is a semisimple algebra,
and hence $\,\H_k \cong \left(\End_{\ms O_k}P\right)^{\mathrm{opp}}$ is a semisimple algebra.

${\it 2 \Rightarrow 3}$. By Proposition~\ref{ind}, the standard modules $ M(\tau) $
in $\ms O_k$ are indecomposable. Hence, if $\ms O_k$ is semisimple, then all $M(\tau)$ are
simple, and we have $ L(\tau)=M(\tau) $ for all $\tau\in\W$. Now, suppose that $0\ne
I\subset H_k$ is a proper two-sided ideal. $H_k$ and $I$ are torsion
free over $\c[V]$. Therefore $0\ne \Ireg\subset\Hreg$ is a two-sided
ideal of $\Hreg=\D W$, which is a simple algebra. Hence, $\Ireg=\Hreg$.

Now, we can always find a primitive ideal $J\subset H_k$, containing
$I$.  By \cite{Gi}, Theorem 2.3, every primitive ideal is the
annihilator of some simple module in $\ms O$. Therefore,
$I\subset\mathrm{Ann}_{H_k}L(\tau)$ for some $\tau$. But
$\Ireg=\Hreg$ implies that $I\cap\c[V]\ne 0$, while
$\mathrm{Ann}_{H_k}L(\tau)\cap\c[V]=0$ because $L(\tau)=M(\tau)$ is
torsion free over $\c[V]$. Contradiction.

${\it 3 \Rightarrow 2}$.
Assuming $H_k$ is simple, we get that
$\mathrm{Ann}_{H_k}L(\tau)=0$ for all $\tau\in\W$. Then
$L(\tau)_{\mathrm{reg}}$ must be nonzero. Indeed, otherwise
$L(\tau)$ would be annihilated by some power of $\delta$, which
contradicts $\mathrm{Ann}_{H_k}L(\tau)=0$.

Thus, $L(\tau)_{\mathrm{reg}}\ne 0$ for all $\tau$. In that case,
each $L(\tau)$ is a submodule of some standard module, by
\cite{GGOR}, Proposition 5.21. By \cite{DO}, 2.5, we have
$[M(\tau):L(\tau)]=1$ and it follows that $L(\tau)\subset M(\tau)$
only if both are the same. Hence, if $L(\tau)\ne M(\tau)$ then it
must be a submodule of some $M(\sigma)$ with $\sigma\ne\tau$.

By {\it loc.cit.}, we can order the elements $\tau_1< \ldots <
\tau_d$ of $\W$ in such a way that the matrix with the entries
$[M(\tau_i):L(\tau_j)]$ is upper-triangular. From the previous
paragraph it follows that if $L(\tau_i)\ne M(\tau_i)$ then the
$i$-th column of this (upper-triangular) matrix has at least one
nonzero off-diagonal entry. This gives us immediately that
$M(\tau_1)=L(\tau_1)$ is simple. Therefore,
$[M(\tau_1):L(\tau_2)]=0$, which implies that $L(\tau_2)=M(\tau_2)$
is simple, and so on. As a result, we conclude that
$L(\tau_i)=M(\tau_i)$ for all $i$, i.e. all standard modules are
simple. Now, the BGG reciprocity (see \cite{GGOR}, Section~2.6.2 and
Proposition~3.3) implies that each $L(\tau)=M(\tau)$ is projective
and $\ms O$ is semisimple (cf. the concluding remark of \cite{BEG},
Section 2).
\end{proof}
\begin{remark}
\la{Rem22}
Theorem~\ref{ss} can be refined by adding that $\, \O_k \simeq \mod(\H_k) \,$
whenever one of its three equivalent conditions holds. Indeed, if (say) $ \O_k$
is semisimple, then
$ M(\tau) = L(\tau) $ for all $ \tau \in \W $. This implies that
$ L(\tau)_{\mathrm{reg}}=M(\tau)_{\mathrm{reg}}\ne 0$ for all $ \tau \in \W $,
since each $ M(\tau) $ is torsion free over $\c[V]$. Now,
every object $M$ in $ \O_k$ can be filtered as in Proposition~\ref{ind}, so
$\, \Mreg\ne 0 \,$ if $ M \ne 0 $ in $ \O_k $. Thus, $ \O_k$ being semisimple
implies that $\,\Otor=0$, and the equivalence $\, \O_k \simeq \mod(\H_k) \,$ follows
then from Proposition~\ref{tor}.
\end{remark}
\begin{remark}
The implication ``${\it 2 \Rightarrow 3}$'' holds without
the assumption $\dim\H_k=|W|$, since the KZ functor
is not used in the proof. This implication is also equivalent to
\begin{equation*} L(\tau)=M(\tau)\ ,\quad
\forall\tau\in\W\quad\Rightarrow\quad H_k\ \text{is a simple ring}\ ,
\end{equation*}
which is one of the key observations of \cite{BEG} (see Section 3 of {\it
loc.cit.}).
\end{remark}

\medskip

We now call a multiplicity vector $\,k=\{k_{C,i}\} \in \c^{\sum_{C\in\ms A/W}n_C}\,$
{\it regular} if the category $\ms O_k(W) $ is semisimple.
Write $\,\Reg(W) \,$ for the subset of all
regular vectors in $\,\c^{\sum_{C\in\ms A/W}n_C} \,$.
In view of Theorem~\ref{ss}, for those groups $W$ where it is known that
$\dim\,\H_k=|W|$, $\,\Reg(W) $ coincides with the set of all $k$'s for
which the Hecke algebra $\H_k(W)$ is semisimple and the Cherednik algebra $ H_k(W) $
is simple. In general, we will need the following fact.
\begin{lemma}\la{r}
For any group $W$, $\,\Reg(W)\,$ is a connected subset in $\c^{\sum_{C\in\ms
A/W}n_C}$.
\end{lemma}
\begin{proof}
Put
\begin{equation}\label{z}
z(k)=\sum_{H\in\ms A}\sum_{i=0}^{n_H-1}n_Hk_{H,i}e_{H,i}\in\c W\,.
\end{equation}
The element $z(k)$ is central in $\c W$, hence it acts on each
$\tau\in\W$ as a scalar, which we denote by $c_\tau(k)$. Obviously, $c_\tau(k)$
is a linear function of $k$. Moreover, according to \cite{DO}, Lemma
2.5, $c_\tau(k)$ is a linear function with {\it nonnegative integer}
coefficients. By {\it loc.cit.}, Proposition 2.31, $M(\tau)$ is
simple if $\,c_\sigma(k)-c_\tau(k)\notin\N\,$ for all $\sigma\in\W$.
Hence, if $k$ is generic, namely,
\begin{equation}\label{gen}
c_\sigma(k)-c_\tau(k)\notin\N\ ,\quad\forall \sigma,\tau\in\W\,,
\end{equation}
then all standard modules are simple and, as in the proof of Theorem
\ref{ss}, the category $\ms O_k$ is semisimple.

It follows that the complement to $\Reg(W)$ is contained in a locally
finite union of hyperplanes, thus $\Reg(W)$ itself is connected.
\end{proof}

\section{Shift Functors and KZ Twists}
\la{SF}

\subsection{Shift functors}
\la{sfff}
Recall that $ \O_k $ is the full subcategory of $ \Mod(H_k)$ consisting
of finitely generated modules on which the elements
$\, \xi \in V \,$ act locally nilpotently.
It is convenient to enlarge $ \O_k $ by dropping the finiteness
assumption: following \cite{GGOR}, we denote the corresponding category
by $ \Oln_k $. The inclusion functor $\,\Oln_k \into \Mod(H_k) \,$ has
then a right adjoint $\,\r_k:\,\Mod(H_k) \to \Oln_k \,$, which assigns
to $ M \in \Mod(H_k) $ its submodule
$$
\r_k(M) := \{m \in M\ :\  \xi^d m = 0\ ,\ \forall \, \xi \in V\,,\ d \gg 0\}\ .
$$
Thus,  $\, \r_k(M) $ is the largest submodule (i.e., the sum of all submodules) of $ M $
belonging to $ \Oln_k $. When restricted to finitely generated modules,
$ \r_k $ defines a functor $ \mod(H_k) \to \O_k $; however, $\,
\r_k(M) \not\in \O_k $ for an arbitrary $ M \in \Mod(H_k) $.

We will combine $ \r_k $ with localization to define functors between
module categories of $ H_k $, with different values of $ k $.
To this end, for each $ k $, we identify $\, H_k[\delta^{-1}] = \D W \,$ using
the Dunkl representation (see Proposition~\ref{Loc}) and write $\,\theta_k:\, H_k \to \D W \,$
for the corresponding
localization map. Associated to $ \theta_k $ is a pair of natural functors: the localization
$\,(\theta_k)^*:\,\Mod(H_k) \to \Mod(\D W)$,$\, M \mapsto \D W \otimes_{H_k} M$, and its right adjoint -- the
restriction of scalars $\,(\theta_k)_*:\,\Mod(\D W) \to \Mod(H_k) \,$  via $ \theta_k $.
Given a pair of multiplicities, $k$ and $ k'$ say, we now define
\begin{equation*}
\T_{k \to k'} := \r_{k'} \, (\theta_{k'})_* \, (\theta_k)^*: \ \Mod(H_k) \to \Mod(H_{k'}) \ .
\end{equation*}
\begin{prop}\la{lmax}
The functor $ \T_{k \to k'} $  restricts to a functor: $\, \O_k \to \O_{k'}\,$.
\end{prop}
\begin{proof}
Given $ M \in \O_k $, let $\, N := (\theta_{k'})_* (\theta_k)^* M
\in \Mod(H_{k'}) $. To prove the claim we need only to show that $
\r_{k'}(N) $ is a finitely generated module over $ H_{k'}$. Assuming the contrary,
we may construct an infinite {\it strictly} increasing chain
of submodules $\,N_0 \subset N_1 \subset N_2 \subset\,
\ldots \subset \r_{k'}(N)\subset \Mreg$, with $N_i\in \O_{k'}$.
Localizing this chain, we get an infinite chain of
$\Hreg$-submodules of $\Mreg$. Since $ \Mreg $ is finite over $\c[\vreg]$ and $\c[\vreg]$ is
Noetherian, this localized chain stabilizes at some $i$. Thus, omitting
finitely many terms, we may assume
that $(N_i)_{\mathrm{reg}}=(N_0)_{\mathrm{reg}}$ for all $i$. In
that case all the inclusions $\,N_i \subset N_{i+1} \,$ are
essential extensions, and since each $\, N_i \in \O_{k'} \,$, the
above chain of submodules can be embedded into an injective hull of
$ N_0 $ in $ \O_{k'} $ and hence stabilizes for $ i \gg 0 $. (The
injective hulls in $ \O_{k'} $ exist and have finite length, since $
\O_{k'} $ is a highest weight category, see \cite{GGOR},
Theorem~2.19.) This contradicts the assumption that the inclusions
are strict. Thus, we conclude that $ \r_{k'}(N) $ is finitely generated.
\end{proof}

\begin{defi}
We call $\, \T_{k \to k'}: \,\O_k \to \O_{k'} \,$ the {\it shift
functor} from $ \O_k $ to $ \O_{k'}$.
\end{defi}

The following lemma establishes basic properties of the functors $\,\T_{k \to k'}\,$.
\begin{lemma}\la{fu}
Let $\,k, k', k''\,$ be arbitrary complex multiplicities,
and let $\, M \in \O_k $.

$\mathsf{(i)}$\ If $k\in\Reg$, then $\T_{k \to k}(M)\cong M $.

\medskip $\mathsf{(ii)}$\ If $\,k, k'\in\Reg\,$ and $ M $ is simple, then $\,\T_{k \to k'}(M) $ is either simple or zero.

\medskip $\mathsf{(iii)}$\ If $\,k \in \Reg\,$ and $M$ is simple with $\,\T_{k \to k'}(M)\ne 0\,$, then
$$
[\T_{k' \to k''} \circ \T_{k \to k'}](M) \cong \T_{k \to k''}(M)\ .
$$
\end{lemma}
\begin{proof}
To simplify the notation, we will write $ \Mreg $ for both
$\, (\theta_k)^* M \in \Mod(\D W) \,$ and
$\, (\theta_{k'})_* (\theta_k)^* M \in \Mod(H_{k'})\,$
whenever this does not lead to confusion.

$\mathsf{(i)}$ For regular $k$, $\,\Otor=0\,$, hence
$\,\Mreg \ne 0\,$ whenever $\, M \ne 0 \,$, and $\,M \,$ is
naturally an $H_k$-submodule of $ \Mreg$. We need to show
that $M$ is the maximal submodule of $ \Mreg$ belonging
to $ \O_k $. If $\, M \subseteq N \subset \Mreg\,$,
with $\,N\in \O_k\,$, then $\,N_{\mathrm{reg}}= \Mreg\,$.
Since $\Otor=0$, this forces $\,N=M\,$, proving $\mathsf{(i)}$.

$\mathsf{(ii)}$ For regular $\,k\,$, the simple objects in $ \O_k $ are the standard modules
$M(\tau)$. If $ M = M(\tau) $ is such a module, then $\Mreg$
is a simple $\D W$-module. Hence, if $\, 0\ne N\subseteq (\theta_{k'})_* \,(\theta_k)^*(M) \,$,
then $\,N_{\mathrm{reg}}=\Mreg\,$. As a result,
if $\,0\ne N \subsetneq N'\subset (\theta_{k'})_* \,(\theta_k)^*(M)\,$ are two
submodules $\, N,\, N' \in \O_{k'} \,$, then $\,N_{\mathrm{reg}}=N'_{\mathrm{reg}}\,$
and $\,(N'/N)_{\mathrm{reg}}=0\,$.
But this contradicts the fact that $(\O_{k'})_{\rm tor}=0 $.
Thus $\,(\theta_{k'})_* \,(\theta_k)^*(M)\,$ may have at most
one nontrivial submodule $\, N\in \O_{k'}$ which,
therefore, must be simple.

$\mathsf{(iii)}$ If $\,M\in \O_k\,$ is simple, then $\Mreg$ is simple.
Hence, if $\,N= \T_{k \to k'}(M) \ne 0\,$, then $\,N_{\mathrm{reg}}=\Mreg\,$,
and therefore $ \r_{k''}(\Mreg) = \r_{k''}(N_{\mathrm{reg}})$.
\end{proof}
\begin{remark}
\la{en}
Part $\mathsf{(ii)}$ of Lemma~\ref{fu} can be restated as follows:
if $\,k,\,k' \in \Reg\,$, then $\, \T_{k \to k'}: \, \O_k \to \O_{k'} \,$ transforms
standard modules either to standard modules or zero.
\end{remark}
\begin{cor}\la{eq}
Assume that $\,k,\,k'\in\Reg\,$. Then the following are equivalent:

$(1)$\ $\T_{k \to k'}[M_k(\tau)] \cong  M_{k'}(\tau')\,$,

$(2)$\ $ M_k(\tau)_{\mathrm{reg}} \cong
M_{k'}(\tau')_{\mathrm{reg}}\,$ as $\Hreg$-modules.

\end{cor}

\begin{proof}
$(1) \,\Rightarrow \, (2)$. Let $\,M=M_k(\tau)\,$. Since
$\,k\in\Reg\,$, $ M $ is a simple $H_k$-module and $\Mreg$ is a
simple $\Hreg$-module. Then, if $(1)$ holds,
$\,M_{k'}(\tau')_{\mathrm{reg}}\,$ is a submodule of a simple module
$\Mreg$, and hence $\,M_{k'}(\tau')_{\mathrm{reg}}=\Mreg\,$, as
needed.

$(2) \,\Rightarrow \, (1)$. If $(2)$ holds, $\Mreg$ contains a copy
of $\,M_{k'}(\tau')\,$. Lemma~\ref{fu}$\mathsf{(ii)}$
then implies that $\,M_{k'}(\tau')\cong \T_{k \to k'}(M)$.
\end{proof}

\subsection{KZ twists}
\la{kzt}

Throughout this section we assume that $k_{H,i}\in\Z$. In that case
the Hecke algebra $\H_k $ is isomorphic to the group algebra $ \c W$,
so that $\dim \H_k = |W|$. We can use the results of the previous
section, which we summarize in the following
\begin{prop}\la{ssc}
If $ k $ is integral, then the algebra $H_k$ is simple,
the category $\ms O_k$ is semisimple,
all standard modules $M_k(\tau) \in \O_k $ are irreducible,
and the functor $ \KZ $ is an equivalence:
$\,\O_k\overset{\sim}\rightarrow \mod(\c W)$.
\end{prop}
\begin{proof} The first two claims follow from Theorem \ref{ss}.
The irreducibility of $M(\tau)$ then
follows from the fact that these modules are indecomposable. Finally,
$L_k(\tau)=M_k(\tau)$ implies that $\Otor=0$ (see Remark~\ref{Rem22}),
so the last claim is a consequence of Proposition~\ref{tor}.
\end{proof}

Now, applying the KZ functor to $M_k(\tau)$, we see that,
for integral $k$'s, any local solution to the KZ system \eqref{kz} is a
global single-valued function $y: \vreg\to \tau$. Thus we have the
following result, due to Opdam.
\begin{prop}[see \cite{O1,O}]\la{mon} If $ k $ is integral, every
local solution of the system \eqref{kz} extends to a rational function on $V$,
with possible poles along $H\in\ms{A}$. The monodromy of this
system on $\vreg/W$ is given by the $W$-action $^wy:=wyw^{-1}$ on
the space of global solutions.
\end{prop}

\begin{remark}\la{sv}
If $\{e_i\}$ is a basis of $\tau$, then any global solution of
\eqref{kz} can be written in the form $y_i=\sum f_{ij}\otimes e_j$, with
$f_{ij}\in\c[\vreg]$. Since $\{y_i\}$ are linearly independent at
each point $x\in\vreg$, the matrix $\,\|f_{ij}\|\,$ is
invertible, with inverse matrix $\,\|f_{ij}\|^{-1}$ having entries in $\c[\vreg]$.
\end{remark}

Next, the last statement of Proposition \ref{ssc} implies that the functor
$ \KZ $ induces a bijection between the simple objects of $\ms O_k$ and
$\mathrm{mod}(\c W)$, i.e. between the sets $\{M_k(\tau)\}_{\tau\in\W}$
and $\W$. For any integral $k$, this defines a permutation
\begin{equation*}
\kz_k:\,\W \to \W\ ,\quad \kz_k(\tau) := \KZ\,[M_k(\tau)]\ ,
\end{equation*}
which we call a {\it KZ twist}. It is obvious from the definition that $\kz_0(\tau)=\tau$ for
all $\tau$. It is also clear that $\kz_k$ preserves dimension.

As mentioned in the Introduction, our aim is to establish the following {\it additivity property} of KZ twists:
\begin{equation}
\la{addit}
\kz_k\circ\kz_{k'}=\kz_{k+k'}\ ,\quad \forall\, k,\,
k'\in \Z^{\sum_{C\in\ms A/W}n_C} \ ,
\end{equation}
which was first proved (under the assumption that $\, \dim\, \H_k = |W| \,$) in \cite{O1,O}.
We begin by relating $\kz_k$ to localization in the category $\O $.
\begin{prop}\la{loc} If $k$ is integral, there is an isomorphism of $\Hreg$-modules
$$
M_k(\tau)_{\mathrm{reg}}\cong M_0(\sigma)_{\mathrm{reg}}\ ,
$$
where $\sigma=\kz_{k}(\tau)$ and $M_0(\sigma)$ is the standard module
over $H_0=\D(V)*W$ corresponding to $ \sigma $.
\end{prop}

\begin{proof} Choose a basis $\{e_i\}$ of $\tau$, and let $M=M_k(\tau)$.
By \ref{lkz} and \ref{kzt}, we have a flat connection $\partial$ on
$\Mreg \cong \c[\vreg]\otimes\tau$, and a space $\sigma$ of the
horizontal sections of this connection, with a basis $y_i=\sum
f_{ij}\otimes e_j$. The action of $W$ on $\sigma$ is given by
\begin{equation*}
    wy_i=\sum f_{ij}\circ w^{-1}\otimes we_j\,=\, ^wy_i\,,
\end{equation*}
that is, it coincides with the monodromy of the connection, cf.
Proposition~\ref{mon}. Thus there is a subspace $\sigma\subset\Mreg$
which is isomorphic to $\kz_k(\tau)$ as a $W$-module and
such that
$\partial_\xi \sigma=0$ for all $\xi\in V$. Also, by
Remark \ref{sv}, we have
\begin{equation*}
    \c[\vreg]\cdot\sigma=\c[\vreg]\cdot\tau=\Mreg\,.
\end{equation*}
It follows that $ \Mreg \cong M_0(\sigma)_{\mathrm{reg}}\,$, with
$\sigma=\kz_k(\tau)$, as required.
\end{proof}
Taking $\,\tau'=\kz_{k'}^{-1}\circ \kz_k(\tau)\,$ in Proposition \ref{loc}, we get
\begin{cor}\la{lloc}
For any integral $k, k'$ there is a permutation $\tau\mapsto\tau'$
on  $\,\W\,$, such that $\,M_k(\tau)_{\mathrm{reg}} \cong M_{k'}(\tau')_{\mathrm{reg}}\,$
for all $\tau\in\W$.
\end{cor}

Now, we are in position to state the main result of this section.
\begin{theorem}
\la{zero}
Let $k$ and $k'$ be complex multiplicities such that $\,
k'_{H,i} - k_{H,i} \in \Z \,$ for all $ H $ and $i$. Then

$(1)$\ $\T_{k \to k'}(M)\ne 0$ for any standard module $\,M = M(\tau) \in \O_k\,$.

$(2)$ If $\,k, k' \in \Reg $, then $\, \T_{k \to k'} $ maps $\, M_k(\tau) \,$ to $\, M_{k'}(\tau') \,$, with
$\,\tau' = \kz_{k-k'}(\tau)\,$. 
\end{theorem}
Before proving Theorem~\ref{zero} (see Section~\ref{pzero} below),
we deduce some of its implications. First, Theorem~\ref{zero}
implies the additivity property \eqref{addit} of KZ twists.
\begin{cor}[Conjecture in \cite{O,O3}]\la{Opcon}
The map \, $k\mapsto \kz_k$
is a homomorphism from the additive group of integral multiplicities
to the group of permutations on $\W$.
\end{cor}
Indeed, all integral values of $k$ are regular, so by Theorem
\ref{zero} and Lemma \ref{fu}($\mathsf{iii}$),
\begin{equation*}
\T_{0 \to k+k'}[M_0(\tau)] \cong (\T_{k \to k+k'}\circ
\T_{0 \to k}) [M_0(\tau)]\ .
\end{equation*}
Hence $\kz_{k+k'}(\tau)=[\,\kz_{k'}\circ \kz_{k}](\tau)$, as
required. {\hfill $\square$}

Next, we will prove one of the key results for describing the structure of
quasi-invariants
in Section~\ref{MQ}. For this, recall the module $\QQ_k(\tau)$ defined in
Section \ref{ttau}: by construction, this is a submodule
of $ \c[\vreg]\otimes\tau $ under the differential action of $ H_k $.
Using notation of Section~\ref{sfff},
we now identify $\,\c[\vreg]\otimes\tau \,$ with $\,(\theta_k)_*(\theta_0)^*(M)\,$,
where $\,M=M_0(\tau)$.
The Dunkl operators $T_{\xi,k}$ act on $\,\c[\vreg]\otimes\tau \,$ by lowering the
degree. Together with property
\eqref{sa0}, this implies that $\QQ_k(\tau) \in \O_k $. Lemma \ref{fu} shows then
\begin{equation}
\la{shf00}
\QQ_k(\tau) \cong \T_{0 \to k}[M_0(\tau)]\ .
\end{equation}
On the other hand, by Proposition~\ref{loc}, we have
\begin{equation}
\la{shf0}
\T_{0 \to k}[M_0(\tau)] \cong M_k(\tau')\ ,\quad\
\tau= \kz_k(\tau')\ .
\end{equation}
Combining \eqref{shf00} and \eqref{shf0}, we arrive at the following conclusion.
\begin{prop}\la{sq} There is an isomorphism of $H_k$-modules
$\, \QQ_k(\tau) \cong M_k(\tau')$, where $\,\tau =\kz_{k}(\tau')$.
\end{prop}

\begin{remark}
\la{Rdefq}
Formula \eqref{shf00} suggests a conceptual way to define quasi-invariants with
values in an arbitrary $W$-module $ \tau $ (cf. Section~\ref{ttau}). Specifically,
for any $\, k = \{k_{H,i}\}\,$, with $ k_{H,i} \in \Z $, the module $ \QQ_k(\tau) $ can
be described by
$$
\QQ_k(\tau) = \{\varphi \in \c[\vreg] \otimes \tau\ :\ \theta_k(\xi)^d \varphi = 0\ ,\ \forall\, \xi \in V\ ,
\ d \gg 0\}\ ,
$$
where $\,\theta_k\,:\,H_k \into \D W \,$, and $ \D W $ operates on
$\,\c[\vreg] \otimes \tau \,$ via the identification  $\,\c[\vreg] \otimes \tau \cong
(\D W/J) \otimes_{\c W} \tau \,$, by formulas \eqref{acts}.

\end{remark}

\subsection{Proof of Theorem \ref{zero}}
\la{pzero}
We first prove the result for integral $\,k,\,k'\,$ and then use a deformation argument in $k$.
We begin with some preparations. Given $ M=M_k(\tau) \in \O_k $, we
identify $\,\Mreg \cong \c[\vreg]\otimes\tau \,$ as a $ \c[\vreg]*W $-module.
The action of $\partial_\xi $ gives then a flat connection on
$\, \c[\vreg]\otimes\tau \,$, depending on $k$, which is the KZ connection \eqref{conn}.
The algebra $H_{k'}$ also acts on $\Mreg$, with $\xi\in V$ acting as the Dunkl operator
\begin{equation}
\label{du1} T_{\xi, k'}=
\partial_\xi-\sum_{H\in\ms{A}}
\frac{\alpha_H(\xi)}{\alpha_H}\sum_{i=0}^{n_H-1}n_H k'_{H,i}
\e_{H,i}\,,
\end{equation}
where $\partial_\xi$ acts by formula \eqref{conn}. Clearly, for
$\,k'=k+b\,$ with $b$ fixed, the action of both $T_{\xi,k}$ and
$T_{\xi,k'}$ on $\Mreg=\c[\vreg]\otimes\tau$ depends polynomially on
$k$.

Recall that $\c[\vreg]$ is obtained from $\c[V]$ by inverting the
homogeneous polynomial $\delta$, so the standard grading on $\c[V]$ extends naturally
to a $\Z$-grading on $\c[\vreg]$ and $\Mreg$.

Now, we choose dual bases $\{\xi_i\}$ and $\{x_i\}$ in $V$ and $V^*$,
and, following \cite{DO}, consider the (deformed) Euler operator
\begin{equation}\label{eu}
E(k) := \sum_{i} x_iT_{\xi_i, k}\in\D W\ .
\end{equation}
It is easy to see that $\,E(k)=E(0)-z(k)\,$, with $E(0)=\sum_{i}
x_i\partial_{\xi_i}$ and $z(k)$ given by \eqref{z}.
Using formula \eqref{conn} for the action of $\partial_\xi$ on
$\Mreg$, we get
\begin{equation*}
    E(0)(f\otimes v)=E(0)(f)\otimes v +f\otimes z(k)(v)\,.
\end{equation*}
Being a central element in $ \c W $, $\, z(k) $ acts on $\tau\in\W$ as a
scalar $\,c_\tau(k)\,$, so that
\begin{equation}\label{cf}
\mathrm{tr}\, z(k)|_\tau = c_\tau(k)\, \dim\tau \ .
\end{equation}
For any homogeneous $\,f\otimes v\in M_k(\tau)_{\mathrm{reg}}$, we have then
\begin{equation*}\la{eua}
    E(k')(f\otimes v)=(m+c_\tau(k)-z(k'))(f\otimes v)\,,\quad m=\deg f\,.
\end{equation*}
This gives the following result (cf. \cite{DO}, Lemma~2.26).
\begin{lemma}\la{euler}
Let $\sigma\in\W$ and $m\in\Z$. Let $\,M_{\sigma,m} \,$ be a homogeneous subspace
of $ M_k(\tau)_{\mathrm{reg}} $ of degree
$m $, which is isomorphic to $\sigma\in\W$ as a $W$-module.
Then $\,E(k')\,$ acts on $\,M_{\sigma,m}\,$ as
multiplication by $\,m+c_\tau(k)-c_{\sigma}(k')\,$.
\end{lemma}
\noindent
Arguing as in \cite{DO}, Proposition~2.27, from Lemma~\ref{euler} we deduce
\begin{lemma}\la{gra} Every $H_{k'}$-submodule of
$\Mreg=M_k(\tau)_{\mathrm{reg}}$ is graded. With respect to this
grading, the actions of $T_{\xi, k'}$, $W$ and $V^*$ have degrees
$-1$, $0$ and $1$, respectively.
\end{lemma}

\medskip

Now, let us summarize what we have so far in the case of integral $k,k'$.
By Corollary \ref{lloc} and Corollary \ref{eq},
\begin{equation}\la{for}
M_k(\tau)_{\mathrm{reg}}\cong M_{k'}(\tau')_{\mathrm{reg}}\quad
\text{\and}\quad\ms T_{k \to k'}[M_k(\tau)] \cong M_{k'}(\tau')
\end{equation}
for some $\tau'\in\W$. Thus, viewed as a $H_{k'}$-module,
$M_k(\tau)_{\mathrm{reg}}$ contains a (unique) submodule
$\,N \in \O_{k'}\,$, which is isomorphic to $M_{k'}(\tau')$. Note that
both $M_k(\tau)_{\mathrm{reg}}$ and $M_{k'}(\tau')_{\mathrm{reg}}$
are free over $\c[\vreg]$, so the first isomorphism in \eqref{for} implies
that $\dim\tau=\dim\tau'$.
Further, we claim that $\,N \subseteq M_k(\tau)_{\mathrm{reg}}\,$ satisfies
\begin{equation}\label{sa1}
\delta^r M_k(\tau) \subset N \subset \delta^{-r}M_k(\tau)\,,
\end{equation}
where $r\gg 0$ depends on the difference $\,k'-k\,$ but not on $k$.
To see this, we can use Proposition \ref{sq} to identify $M=M_k(\tau)$ with
one of the modules $\QQ_k(\sigma)$. Under such an
identification, $\, N=\ms T_{k',\,k}(M)\,$ gets identified with
$\QQ_{k'}(\sigma)$, and then \eqref{sa1} follows from \eqref{sa}.
Now, \eqref{sa1} and Lemma \ref{gra} show that the subspace $ \tau' $
generating $\, N\,$ sits in $\,\delta^{-r}M_k(\tau)\,$, and its
homogeneity degree $\,\deg\tau'\le r\deg\delta\,$. Thus, summing up,
we have
\begin{lemma}\la{subs}
Assume that $k $ and $ k'$ are integral, and let $\,M := M_k(\tau)\,$,
with $\tau\in\W$. Then $\, \Mreg \,$ contains a subspace $\,\tau' \,$,
such that $\, \dim\tau'=\dim\tau \,$, $\, T_{\xi,k'}(\tau') = 0 \,$ for
all $\,\xi \in V\,$, and
\begin{equation}
\label{tau'}
\tau'\subset \delta^{-r}M\,,\quad \deg\tau'\le r\deg\delta\ ,
\end{equation}
where $r$ depends only on $\,k'-k\,$.
\end{lemma}

\begin{proof}[Proof of Theorem~\ref{zero}]
Let $k$ be arbitrary complex-valued and let $\,k'=k+b\,$, where $b$ is
integral. Throughout the proof we will keep $b$ fixed, while
regarding $k$ as a parameter. As above, we identify $\,M = M_k(\tau)\,$ with one
and the same vector space $\,\c[V] \otimes \tau\,$ for all $k$. The localized
modules $\Mreg$ are then identified with $\,\c[\vreg]\otimes\tau\,$, and the
information about $k$ is encoded in the connection \eqref{conn}.

Let $(\Mreg)^0$ denote the subspace of all elements in $\Mreg$ that
are annihilated by $T_{\xi,k'}$ for all $\xi$. Obviously,
$(\Mreg)^0$ is preserved by the action of $W$. If $\ms W\subset
(\Mreg)^0$ is a $W$-invariant subspace isomorphic to some
$\sigma\in\W$, then we have a nonzero homomorphism from
$M_{k'}(\sigma)$ to $\Mreg$ (by the universality of the standard
modules). Therefore, to see that $\ms T_{k \to k'}M \ne 0$
it suffices to see that $(\Mreg)^0 \ne 0$.

We put on $ \Mreg $ a positive increasing filtration $ \{F_j\}$, with %
\begin{equation*}
    F_j=\{m\in\Mreg\,|\, m=\delta^{-j}u\,,\ \text{where}\ u\in M\
    \text{and}\ \deg u\le 2j\deg\delta\}\,.
\end{equation*}
Each $F_j$ is finite-dimensional, and it is easy to
see that $T_{\xi,k'}F_j\subseteq F_{j+1}$ for all $\xi\in V$.

Set $(F_j)^0:= F_j\cap (\Mreg)^0$, so that
$\, (F_j)^0 = \{m\in F_j\,|\, T_{\xi,k'}(m)=0\, ,\ \forall \,\xi\in V\}\,$.
For each $j\ge 0$, the operators $T_{\xi,k'}$ induce linear
maps between the finite-dimensional spaces $F_j$ and $F_{j+1}$. All
these maps depend polynomially on $k$, and the subspace $(F_j)^0$ is
their common kernel. It follows that $(F_j)^0$ has constant dimension,
independent of $k$, over some dense Zariski open subset in the parameter space.
Now, for integral $k$, we have Lemma \ref{subs}, which says that
$(F_j)^0\ne 0$ for some $j=r$, which depends only on $b=k'-k$.
Therefore, for this particular $j$, $\,(F_j)^0\ne 0$ for
all integral $k$, and hence for all $k$. As a result, $(\Mreg)^0\ne 0$ for all $k$,
which proves the first claim of the theorem. Moreover,
it follows that $\dim (F_j)^0 \ge \dim\tau $ for all $k$.

Recall the set $\Reg$ of regular values of $k$. For a
fixed integral $b$, put $\Reg_b:= \Reg\cap (b+\Reg)$; this is the set
of all $k$ such that both $k$ and $k'=k+b$ are regular. It follows
from Lemma \ref{r} and Theorem \ref{ssc} that the set $\Reg_b$ is
connected and contains all integral points. Since we already know
that $\T_{k \to k'}(M)\ne 0$, Lemma \ref{fu}$\mathsf{(ii)}$
implies that, for $\,k\in\Reg_b\,$, there is a (unique) submodule
$\,N \in \O_{k'}\,$ inside $\Mreg$ (considered as a $H_{k'}$-module).
Moreover, we know that $ N \cong M_{k'}(\tau')$ for some
$\tau'\in\W$. It remains to show that $\tau'$ satisfies
$\tau=\kz_{k'-k}(\tau')$. Note that this is certainly true when
$k=0$, see \eqref{shf0}.

If we regard the generating space $\tau'$ of $N$ as a subspace in
$\Mreg \cong \c[\vreg]\otimes\tau$,
then we know that (1) $\dim\tau'=\dim\tau$, (2) $\tau' = (\Mreg)^0$
for $k\in\Reg_b$, and (3) $\dim (F_j)^0\ge\dim\tau$ for all $k$.
Since $(F_j)^0\subset (\Mreg)^0$, this immediately implies that
$(F_j)^0=\tau'$ for all $k\in\Reg_b$, in particular, it has the same
dimension. Thus, the dimension of $(F_j)^0$ does not jump at any of
the regular values $k\in\Reg_b$, therefore the subspace
$(F_j)^0\subset \c[\vreg]\otimes\tau$ varies continuously with $k$
varying inside $\Reg_b$. As a result, $\tau'=(F_j)^0$ does not
deform as a $W$-module, so it is the same as for $k=0$, in which
case we know already that $\tau=\kz_{k'-k}(\tau')$. This finishes
the proof.
\end{proof}

The above arguments allow us to prove the following property
of KZ twists, which is obtained by a different method in \cite{O3},
Corollary~3.8({\it vi}).
\begin{cor}\la{cs}
If $\tau\in\W$ and $\tau'=\kz_b(\tau)$, then $c_\tau(k)=c_{\tau'}(k)$.
\end{cor}

\begin{proof}
The proof of Theorem \ref{zero} shows that, for all regular
$k$ and $k'=k+b$ ($b$ here is fixed and integral), there is a
homogeneous subspace $\tau'\subset
M_k(\tau)_{\mathrm{reg}}=\c[\vreg]\otimes\tau$ annihilated by all
$T_{\xi, k'}$ and therefore by the Euler operator $E(k')$. This
subspace varies continuously with $k$, hence its homogeneity degree
remains constant. By Lemma \ref{euler}, this degree is given by
$m=c_{\tau'}(k')-c_\tau(k)=c_{\tau'}(b)+c_{\tau'}(k)-c_\tau(k)$.
Therefore $c_{\tau'}(k)-c_\tau(k)$ is constant in $k$,
hence zero.
\end{proof}

\subsection{Heckman-Opdam shift functors}
\la{othershifts}
We briefly explain the relation between our functors $ \T $
and the Heckman-Opdam shift functors introduced in \cite{BEG} and studied in \cite{GS}.

Assume that $k'$ is related to $k$ by \eqref{mrel}, for some $C\in\ms A/W$ and $a=1,\dots,
n_C-1$. Then, by Proposition \ref{shisoa}, we have
\begin{equation*}
 \e H_k \e = \e\, \delta_C^{-a}\,H_{k'}\,\delta_C^a\, \e\,.
\end{equation*}
It follows that$\,\e\, H_{k'}\,\delta_C^a\,\e\,$
is a $ \e H_{k'}\e$-$\e H_{k}\e$-bimodule. Thus, one can define a
functor $\,\ms S_{k\to k'} :\,\Mod(H_{k})\to\Mod(H_{k'})\,$ by
\begin{equation*}
M\mapsto H_{k'}\e\otimes_{\e H_{k'}\e}\e
H_{k'}\delta_C^a\e\otimes_{\e H_{k}\e}\e M\,.
\end{equation*}
It is easy to check that $ \ms S_{k'\to k} $ restricts to a functor from $ \O_{k} $ to $ \O_{k'}$.
Similarly, one defines $\,\ms S_{k'\to k}\,:\,\O_{k'} \to \O_{k}\,$ by
\begin{equation*}
M\mapsto H_{k}\e\otimes_{\e H_k\e}\e
\delta_C^{-a}H_{k'}\e\otimes_{\e H_{k'}\e}\e M\,.
\end{equation*}
Now, checking on standard modules, it is easy to prove
\begin{prop}
\la{compsh}
If $\,k,\,k' \in \Reg\,$, then $\,\T_{k \to k'} \cong \ms S_{k \to k'} \,$.
\end{prop}
In general, however, the functors $ \T$ and $ \ms S $ are not
isomorphic: for example, since $\ms T $ factors through localization, it
always kills torsion (in particular, finite-dimensional modules), while
$\ms S$ does not. On the other hand, if $\,k'=k\,$, then $ \ms S $ is, by definition, isomorphic
to the identity functor, while $\ms T $ is not (for projective objects $\,P \in \O_k\,$,
we still have $\,\T (P) \cong P\,$, by \cite{GGOR}, Theorem 5.3).

We can also define shift functors using shift
operators constructed in Section~\ref{shop}. Briefly,
if $k'-k$ is integral, then there is a $W$-invariant
differential operator $S$ satisfying
\begin{equation}\la{rela}
T_{p, k'}\, \e S = \e S\, T_{p, k}\ ,\qquad\forall\, p\in\c[V^*]^W\,.
\end{equation}
(Such $S$ is a composition of elementary shift operators of
Theorem~\ref{shiftoper}.) Regard $\Hreg$ as a $ H_{k'}$-$H_k$-bimodule
via the Dunkl representation, and consider its sub-bimodule $\, P \,$
generated by $ \e S $, i.e. $\,P := H_{k'}\,(\e S)\,H_k \subset \Hreg\,$.
Now, given $\, M\in\ms O_k$, define $\, M':= P \otimes_{H_k} M\,$.
Clearly, if $\, \Mreg \ne 0 \,$, then $M'$ is a nonzero $H_{k'}$-module embedded in
$\Hreg\otimes_{H_k}M=\Mreg$. To prove that $\,M'\in\ms O_{k'}\,$, it is suffices to check
that $T_{p,k'}$, with $p\in \c[V^*]^W$, act locally nilpotently on $M'$.
But this follows immediately from
\eqref{rela} and the well-known fact that the adjoint action of
$\c[V^*]^W$ on $H_k$ is locally nilpotent. Thus, $\,P \otimes_{H_k} \mbox{---}\,$
defines a functor $\,\O_k \to \O_{k'}\,$. Again, when $\,k, k'\in\Reg\,$, it is easy to
show that this functor is isomorphic to $\,\T_{k\to k'}\,$ (and hence, ${\ms S}_{k \to k'}$,
by Proposition~\ref{compsh}).
We can use this to prove the following useful observation.

\begin{prop}\la{eval} Let $k,k'$ be integral, and let $S$ be a composition of shift
operators of Theorem~\ref{shiftoper}, such that
\begin{equation*}
L_{p, k'}\circ S=S\circ L_{p,k}\ ,\qquad\forall\, p\in\c[V^*]^W\ .
\end{equation*}
Then $\,S [Q_k] \subseteq Q_{k'}\,$, where $Q_k$ and $Q_{k'}$ are
the corresponding modules of quasi-invariants.
\end{prop}
\begin{proof} Using the fact that $\,\T_{k\to k'}(\QQ_k)=\QQ_{k'}\,$ and the above relation
between $ \T $ and $S$, we have
\begin{equation*}
\e (S[Q_k]\otimes 1)=\e S\,[Q_k\otimes 1]= (\e S)[\QQ_k]\subseteq
H_{k'}\e S H_k\otimes_{H_k} \QQ_k\subset \QQ_{k'}=\e (Q_{k'}\otimes
1)\, .
\end{equation*}
Thus $\,S[Q_k]\subseteq Q_{k'}\,$, as required.
\end{proof}
\section{The Structure of Quasi-invariants}
\la{MQ}

\subsection{Cohen-Macaulayness}
\la{CohM}
First, we consider the module of $W$-valued
quasi-invariants $ \QQ_k $ introduced in Section~\ref{nqi}. By
\eqref{ssum}, this is a $H_{k}\otimes \c W$-module, which can be
decomposed as
\begin{equation*}
\QQ_k = \bigoplus_{\tau\in \W} \QQ_k(\tau)\otimes\tau^*\ ,
\end{equation*}
with $\,\QQ_k(\tau) \subset \mathbb C[V]\otimes\tau\,$ defined by
\eqref{qctau}. By Proposition \ref{sq}, $\,\QQ_k(\tau) \cong
M_k(\tau')\,$, where $\,\tau'=\kz_{-k}(\tau)$. Hence we have
\begin{prop}
The $H_k\otimes \c W$-module $\QQ_k$ has the direct sum
decomposition
\begin{equation}\la{deco1}
\QQ_k \cong \bigoplus_{\tau\in\W}M_k(\tau')\otimes \tau^*\ ,
\end{equation}
where $\,\tau'=\kz_{-k}(\tau)\,$. In particular, $\,{\QQ}_k\,$ is a
free module over $\c[V]$.
\end{prop}

Now, by Theorem~\ref{Qfat}, the module $ Q_k $ of the usual
quasi-invariants is isomorphic to $\e{\QQ}_k $ as a $\e H_k
\e\otimes \c W$-module. This gives the following result generalizing
\cite{BEG}, Proposition~6.6.
\begin{theorem}
\la{stq} The $\e H_k \e\otimes \c W$-module $Q_k$ has the direct sum
decomposition
\begin{equation}
\label{deco} Q_k \cong \bigoplus_{\tau\in\W}\e
M_k(\tau')\otimes\tau^*\ ,
\end{equation}
where $\,\tau'=\kz_{-k}(\tau)\,$. In particular, $Q_k$ is free over
$\c[V]^W$ and, hence, Cohen-Macaulay.
\end{theorem}
\begin{proof}
The decomposition \eqref{deco} follows directly from \eqref{deco1}.
Each $\,\e M_k(\tau)\,$ is isomorphic to $\,(\c[V] \otimes \tau)^W$
as a $\c[V]^W$-module and, hence, free over $\c[V]^W$. With
\eqref{deco}, this implies the last claim of the theorem.
\end{proof}

\begin{remark}
\la{mistake} Our proof of Theorem \ref{stq} is similar to
\cite{BEG}, however the result is slightly different, because of a
KZ twist. In \cite{BEG}, it was erroneously claimed that
$M_k(\tau)_{\mathrm{reg}} \cong  M_0(\tau)_{\mathrm{reg}}$. By
Theorem~\ref{zero}, this is true only for those groups $W$ and
values of $k$, for which $\kz_k$ is the identity on $\W$. In
general, even in the Coxeter case, there are examples when $ \kz_k $
is non-trivial (see \cite{O1}).
\end{remark}

\subsection{Poincar\'e series} \la{Hser}
Given a graded module $\,M = \bigoplus_{i=0}^\infty M^{(i)}\,$, with
finite-dim\-en\-sional components $ M^{(i)}$, we write
$\, P(M, t) := \sum_{i=0}^\infty t^i\dim M^{(i)}\,$ for the
Poincar\'e series of $M$.
Using Theorem \ref{stq}, we will compute this series for
$Q_k$. Our computation is slightly different from \cite{BEG} as
we begin with $ \QQ_k $.

We equip $\, \c[\vreg]\otimes \tau\,$ with a natural grading, so that
$\,\deg V^*=1\,$ and $\,\deg \tau=0\,$. Each $ {\QQ}_k(\tau) $
is then a graded submodule of $\, \c[\vreg]\otimes \tau\,$, and
by Proposition \ref{sq}, we know that $\,{\bf Q}_k(\tau) \cong M_k(\tau')\,$,
with $\,\tau = \kz_{k}(\tau')\,$. Now, by Lemma \ref{euler} and
Corollary \ref{cs}, the degree of the generating subspace $\tau'$ of ${\QQ}_k(\tau)$
is equal to $\,\deg \tau'=c_{\tau'}(k)=c_\tau(k)$.  Hence
\begin{equation*}
P({\QQ}_k(\tau),\,
t)=(\dim\tau)\,t^{c_{\tau}(k)}\,(1-t)^{-\dim V}\,.
\end{equation*}
As a result, by Proposition~\ref{deco1}, the Poincar\'e series for ${\QQ}_k$ is given by
\begin{equation*}
P({\QQ_k},\, t) = \sum_{\tau\in\W}(\dim\tau)^2\,
t^{c_\tau(k)}\,(1-t)^{-\dim V}\,.
\end{equation*}

Now, to compute $\, P({Q_k},\, t) = P(\e{\QQ}_k,\, t)\,$ we simply  take the
$W$-invariant part of ${\QQ}_k$. This can be done separately for each summand in
\eqref{deco1}. The Poincar\'e series of $\e M_k(\tau)$ is obtained by
multiplying the Poincar\'e series of $(\c[V]\otimes\tau)^W$ by $t^{c_\tau(k)}$.
Hence, writing
\begin{equation}\label{chi}
\chi_\tau(t) := P((\c[V]\otimes\tau)^W,\, t)
\end{equation}
for the Poincar\'e series of $(\c[V]\otimes\tau)^W$, we get
\begin{equation}
\label{hs}
P(\e{\QQ}_k(\tau),\, t)=
t^{c_{\tau'}(k)}\chi_{\tau'}(t)\ ,
\end{equation}
where $\,\tau'=\kz_{-k}(\tau)\,$. Finally, summing up over all
$\,\tau\in\W\,$ as in \eqref{deco}, we find (cf. \cite{BEG})
\begin{equation}
P({Q_k},\, t)=\sum_{\tau\in\W}(\dim\tau) \,
t^{c_\tau(k)}\,\chi_{\tau}(t)\ .
\end{equation}

\subsection{Symmetries of fake degrees}
\la{fdeg} It was pointed out to us by E.~Opdam that the above
results could be used to give another proof of an interesting
symmetry of fake degrees of complex reflection groups (see \cite{O},
Theorem~4.2). Below, we will show that property for the series \eqref{chi};
for the relation of \eqref{chi} to fake degrees we refer the reader
to Opdam's paper  \cite{O}.

Fix a collection of integers $\,a=\{a_C\}_{C\in\ms A/W}\,$, with
$\,a_C\in\{0,1,\dots,n_C-1\}$. Put
$\, \delta_a := \prod_{C\in\ms A/W}(\delta_C)^{a_C}\,$
and write $\,\epsilon_a\,$ for the corresponding one-dimensional
representation of $W$, with character $\,\prod_{C\in\ms
A/W}(\det_C)^{-a_C}$. Now, define $\,k=\{k_{C,i}\}\,$ by $\,k_{C,i} := a_C/n_C\,$ for all $C, i$.
Then, for every $\tau\in\W$, the space $\QQ_{k}(\tau)$ has a simple description:
\begin{equation}\la{expl}
\QQ_{k}(\tau)=\delta_a\,\c[V]\otimes\tau\ ,
\end{equation}
which is easily seen from the definition \eqref{qctau}.

On the other hand, consider $\,k' = g\cdot k\,$, with
$\,g := \prod_{C\in\ms A/W}(g_C)^{a_C}$ and $g_C$ defined by
\eqref{gh}. A straightforward calculation shows that
\begin{equation*}
k'=\sum_{C\in\ms A/W}\sum_{1\le i\le a_C}\ell_{C,n_C-i}\,,
\end{equation*}
where we use the same notation as in Proposition \ref{shiso}.
Now, from Proposition \ref{Gtau}, it follows that
$\,\e\QQ_k(\tau)=\e\QQ_{k'}(\tau)\,$; hence, these two modules
have the same Poincar\'e series. For $\e\QQ_k(\tau)$, we can compute
its Poincar\'e series directly from \eqref{expl}: with notation \eqref{chi},
the result reads $\,t^{\deg\delta_a}\chi_{\epsilon_a\otimes\tau}\,$.
On the other hand, for $\e\QQ_{k'}(\tau)$, we apply \eqref{hs}. Equating
the resulting Poincar\'e series, we get
\begin{equation*}
t^{\deg\delta_a}\chi_{\epsilon_a\otimes\tau}(t)=t^{c_{\tau'}(k')}\chi_{\tau'}(t)\,,\qquad
\tau=\kz_{k'}(\tau')\ ,
\end{equation*}
which is equivalent to \cite{O}, Theorem 4.2.

\section{Appendix: The Baker-Akhiezer Function}
\la{BAF}
When all the multiplicities are integral, the ring of commuting differential operators
$\,\{L_{p,k} \in \D(Q_k)^W : \, p\in\c[V^*]^W\}\,$ has a common eigenfunction $ \psi(\lambda, x) $,
which can be constructed  by applying the shift operators \eqref{s} to the exponential
function $\, e^{\langle \lambda,x\rangle} $. We call such a function the {\it Baker-Akhiezer function};
our goal is to establish basic properties of this function, generalizing
results of \cite{VSC} and \cite{CFV} in the Coxeter case. The most interesting 
property of $ \psi(\lambda, x) $ is  `bispectral' symmetry described in Proposition~\ref{bis}. This property
has been proven for the complex groups of type $ G(m,\,p,\,N) $ in \cite{SvdB}, and
although our proof here is different, the key idea to use the pairing \eqref{pair} is borrowed from \cite{SvdB}.

We restrict ourselves to the case when $\,k_{C,i}\in\Z_{\ge 0}\,$,
with $\,k_{H,0}=0$. In that case, applying successively the
elementary shift operators $ S_k $, see \eqref{s}, produces
a function $\,\psi(\lambda, x)\,$ on $V^*\times V$ of the form
\begin{equation}\la{ps}
\psi(\lambda, x)=P(\lambda, x)\, e^{\langle \lambda, x\rangle}\,,
\end{equation}
where $\langle \lambda, x\rangle$ is the natural pairing, and $\,P\in
\c[V^*\times V]\,$ is a polynomial with leading term
\begin{equation}
\la{p}
P_0=\prod_{C\in\ms{A}/W}
\left(\delta^*_C(\lambda)\,\delta_C(x)\right)^{N_C}\ ,\quad
N_C :={\sum_{i=0}^{n_C-1}k_{C,i}}\ .
\end{equation}
Since the operators \eqref{s} are all homogeneous of degree zero, so is
their composition, and hence $P$ has degree zero with respect to the grading defined by
$\,\deg\,V^* = 1\,$ and $\,\deg\,V = -1\,$.

By construction, $\,\psi$ is a
common eigenfunction of the generalized Calogero-Moser operators
$\,L_{p,k}=\Res\, T_{p,k}\,$:
\begin{equation}\la{cm}
L_{p,k}[\psi] = p(\lambda)\psi\ ,\quad\forall\, p\in\c[V^*]^W\,.
\end{equation}
It is analytic in both variables, and by Proposition \ref{eval}, we have
\begin{equation}\label{psm}
\psi(\lambda, x) \in \QA_k \quad\text{as a function of $x$}\,,
\end{equation}
where $\QA_k$ denotes the analytic completion of the module
of quasi-invariants $Q_k$. Note also that the shift operators in
Theorem \ref{shiftoper} are $W$-invariant, whence
\begin{equation}
\label{sym}
\psi(w\lambda, x)=\psi(\lambda, w x)\ ,\quad\forall\, w\in W\,.
\end{equation}

Now, recall the antilinear isomorphism $*\,:\,V\to V^*$ determined by
the $W$-invariant Hermitian form on $V$, see Section \ref{1.1}. It
is easy to check that $*$ respects the canonical pairing between
$V$ and $V^*$ and is $W$-\,equivariant (see \cite{DO}, Proposition
2.17(i)). It extends to an anti-linear map $\,\c[V^*\times V]\to
\c[V\times V^*]\,$, which we denote by the same symbol. Note that $*$
induces a natural antilinear map $\, *\,:\,\End_\c(\c[V])\to\End_\c(\c[V^*])\,$, and it is easy to check
that
\begin{equation}\la{d*}
(T_{p, k})^*=T_{p^*, \overline k}\,,\quad p\in\c[V^*]\,,
\end{equation}
where $\overline k$ denotes the complex conjugate of $k$ (in our case,
$\overline k=k$). Here, the Dunkl operators on the right
are defined in the same way as $\,T_{p,k}\,$
but with respect to the {\it dual} representation $V^*$ of $W$.

Applying $*$ to $\psi$, we get
\begin{equation}\la{ps*}
    \psi^*(x, \lambda)=P^*(x, \lambda)\,e^{\langle x,\lambda\rangle}\,.
\end{equation}
Let us write $\,\psi=\psi_V(\lambda,x)\,$ to indicate the
dependence of $ \psi $ on the reflection representation $V$ of $W$.
It follows then from \eqref{d*} that $\,\psi^*=\psi_{V^*}(x,\lambda)\,$.
In particular, $ \psi^* $ is a common eigenfunction of the `dual' family
of operators with respect to the $\lambda$-variable:
\begin{equation}\la{cm*}
L_{q,k}[\psi^*] = q(x)\psi^*\ ,\quad\forall \,q\in\c[V]^W\,.
\end{equation}
Now, by `bispectral symmetry' of the Baker-Akhiezer function
we mean the following property.
\begin{prop}
\la{bis}
$\,\psi_V(\lambda, x)= \psi_{V^*}(x, \lambda)\,$.
In particular, $\psi=\psi_V$ is a common solution to the
eigenvalue problems \eqref{cm} and \eqref{cm*}.
\end{prop}

For the proof, we consider
\begin{equation}
\la{phi}
\Phi(\lambda, x):=\sum_{w\in W}\psi(w\lambda, x)=\sum_{w\in
W}\psi(\lambda,wx)\,.
\end{equation}

\begin{lemma}
\la{svo}
The function \eqref{phi} has the following properties:

$(1)$\ $\Phi$ is global analytic in $x$ and $\lambda$;

$(2)$\ $\Phi$ is $W$-invariant in each of the variables, $x$ and $\lambda$;

$(3)$\ $T_{p,k} \Phi = p(\lambda)\Phi\,$ for all $\,p\in\c[V^*]^W$;

$(4)$\ In a neighborhood of $\,\lambda = 0\,$, $\Phi$ admits an expansion
$\,\Phi=\sum_{i} \Phi_i$, where $ \Phi_i \in \c[V]^W\otimes\c[V^*]^W $ is
homogeneous of degree $ i $ in both $\lambda$ and $x$;

$(5)$\ $\Phi(0,x)=\Phi(\lambda, 0)=\Phi(0,0)\ne 0$.

\end{lemma}

\begin{proof}
The first four properties are immediate from the definition;
only $(5)$ needs a proof. Let us define a bilinear map
$\,\c[V^*]\times \c[V] \to \c \,$ by
\begin{equation}
\la{pair}
(p,q)_k := T_{p,k}(q)(0)\ ,\qquad \ p\in\c[V^*]\ ,\
    q\in\c[V]\,.
\end{equation}
This is closely related to the pairing on $\,\c[V]\times
\c[V]\,$ defined in \cite{DO}, which equals $\,(p^*,q)_k\,$ in our notation.

It follows from \cite{DO}, Proposition 2.20 and Theorem 2.18, that
\eqref{pair} is a nondegenerate pairing for any $\,k\in\Reg\,$ satisfying
\begin{equation}\la{conjug}
 (p,q)_k=\overline{(q^*,p^*)_{\overline k}}\ ,\qquad\forall\, (p,q)\in
 \c[V]\times\c[V^*]\ .
\end{equation}
(For integral $k$, we have $\,\overline k = k\,$.) Moreover, by
Proposition 2.17(iii) of {\it loc.cit}, the restriction of $\,(-,-)_k\,$ to $W$-invariants
is also nondegenerate.

We need to prove that $\Phi_0=\Phi(0,0)\ne 0$. Assuming the
contrary, let us take the first nonzero term $\Phi_i$. Then,
substituting the expansion $\,\Phi = \sum_i \Phi_i \,$ into the equations $(3)$, we
see that $\,T_{p,k}\Phi_i=0\,$ for all $p\in\c[V^*]^W$. This implies that
\begin{equation}
\la{nul}
(p,\Phi_i)_k=0\ ,\quad\forall\, p\in\c[V^*]^W\ .
\end{equation}
Note that $\Phi_i$ is $W$-invariant as a function of $x$. Thus,
\eqref{nul} contradicts the nondegeneracy of $(-,-)_k$ and proves
that $\Phi_0=\Phi(0,0)\ne 0$.
\end{proof}
\begin{proof}[Proof of Proposition \ref{bis}]
We can normalize $\Phi$ so that $\Phi(0,0)=1$. Taking a homogeneous
basis $\{p_i\}$ of $\c[V^*]^W$, with $0=\deg p_0\le\deg p_1\le\deg
p_2\le\dots$, we can expand $\Phi$ (as a function of $\lambda$) into
a series in $p_i\,$:
\begin{equation*}
    \Phi(\lambda,x)=\sum_{i\ge 0}p_i(\lambda)q_i(x)\ ,\quad\text{with some $\,q_i\in\c[V]^W$}\ .
\end{equation*}
Evaluating both sides of $\,T_{p,k}\Phi=p(\lambda)\Phi\,$ at $x=0$,
we conclude that the elements $q_i$ form the basis dual
to $\{p_i\}$ with respect to the pairing \eqref{pair}.

If $\{p_i\}$ and $\{q_i\}$ are dual bases, then so are $\{q^*_i\}$
and $\{p^*_i\}$, by \eqref{conjug}. Therefore, we also have
\begin{equation*} \Phi(\lambda, x)=\sum_{i\ge
0}p_i(\lambda)q_i(x)=\sum_{i\ge
0}q^*_i(\lambda)p^*_i(x)=\overline{\Phi(x^*, \lambda^*)}\,.
\end{equation*}
Using the definition \eqref{phi} of $\Phi$, and the fact that
$\langle \mu, x\rangle=\overline{\langle x^*, \mu^* \rangle}$, we
easily conclude that $\psi(\lambda, x)=\overline{\psi(x^*,
\lambda^*)}=\psi^*(x,\lambda)$, which finishes the proof.
\end{proof}

Thus, the properties of $\psi$ in $x$ (say) mirror those in $\lambda$,
but with $V$ replaced by $V^*$. For instance, letting
$\,Q_k^* := Q_k(W, V^*)\,$, we have a counterpart of \eqref{psm}:
\begin{equation}
\label{psl}
\psi(\lambda,x)\in \QA_k^* \quad\text{as a function of
$\lambda$}\ .
\end{equation}
Having this, we can now characterize, similarly to \cite{VSC}, the
Baker-Akhiezer function $\psi(\lambda, x)$ as a {\it unique} function
satisfying \eqref{ps}, \eqref{p} and \eqref{psl}. Furthermore, we
get the following result, which for a Coxeter group $W$ was first
established in \cite{VSC} (see also \cite{CFV}). Recall the
subalgebra $A_k\subset \c[V]$, see \eqref{acc}, and denote by
$A_k^*\subset \c[V^*]$ its `dual' counterpart related to $Q_k^*$.

\begin{prop}
For any $p\in A_k^*$, there exists a differential operator
$L_p\in\D(\vreg)$ in the $x$-variable, with a constant principal
symbol $p$, such that $L_p\psi=p(\lambda)\psi$. The operators
$\{L_p\}_{p\in A_k^*}$ pairwise commute and generate a subalgebra of
$\D(\vreg)$, isomorphic to $A_k^*$.
\end{prop}
Note that, by bispectral symmetry, we also have a similar
commutative subalgebra of differential operators in the `spectral'
variable $\lambda$.

\bibliographystyle{amsalpha}

\end{document}